\pdfoutput=1
\documentclass[]{amsart}
\PassOptionsToPackage{final}{graphicx}
\PassOptionsToPackage{final,hidelinks}{hyperref}
\usepackage{nima}
\usepackage{tikz}
\usepackage{amsaddr}
\usetikzlibrary{arrows}
\usetikzlibrary{decorations.markings}
\usetikzlibrary{decorations.pathmorphing}
\usetikzlibrary{calc}

\title{Quadric Complexes}
\author{Nima Hoda}
\address{D{\' e}partement de math{\' e}matiques et applications \\
  {\' E}cole normale sup{\' e}rieure, 45 rue d'Ulm, 75005 Paris, France
  \\ \ \\ Instytut Matematyczny,
  Uniwersytet Wroc\l awski\\
  pl.\ Grun\-wal\-dzki 2/4,
  50--384 Wroc{\l}aw, Poland}
\email{nima.hoda@mail.mcgill.ca}
\thanks{During the course of this research and the writing of this
  paper, the author was partially funded by an NSERC CGS M, an ISM
  scholarship, a Lorne Trottier fellowship and by the ERC grant
  GroIsRan.}
\date{\today}
 % Distance in graphs

\newcommand{\sk}[2]{#2^{#1}}
\newcommand{\ff}[1]{\overline{#1}}
\DeclareMathOperator{\scC}{C}
\DeclareMathOperator{\scT}{T}
\DeclareMathOperator{\cftf}{C(4)-T(4)}
\DeclareMathOperator{\CAT}{CAT}

\newcommand{\p}{\text{P}}
\newcommand{\catname}[1]{{\normalfont\textbf{#1}}}
\newcommand{\bfslt}{\prec}
\newcommand{\bfsle}{\preceq}
\newcommand{\ucov}[1]{\widetilde{#1}}
\newcommand{\mcS}{\mathcal{S}}
\newcommand{\mcR}{\mathcal{R}}

\tikzset{
  vertex/.style={circle,minimum size=0.15cm,inner sep=0,fill=black},
  thickeraser/.style={line width=2.4pt, white}
}

\begin{document}
\begin{abstract}
  Quadric complexes are square complexes satisfying a certain
  combinatorial nonpositive curvature condition.  These complexes
  generalize \texorpdfstring{$2$-dimensional}{2-dimensional}
  \texorpdfstring{$\CAT(0)$}{CAT(0)} cube complexes and are a square
  analog of systolic complexes.  We introduce and study the basic
  properties of these complexes.  Using a form of dismantlability for
  the \texorpdfstring{$1$-skeleta}{1-skeleta} of finite quadric
  complexes we show that every finite group acting on a quadric
  complex stabilizes a complete bipartite subgraph of its
  \texorpdfstring{$1$-skeleton}{1-skeleton}.  Finally, we prove that
  \texorpdfstring{$\cftf$}{C(4)-T(4)} small cancellation groups act on
  quadric complexes.
\end{abstract}

\maketitle

The study of groups acting on combinatorially nonpositively curved
spaces has been an ongoing theme in group theory tracing its origins
to Dehn's study of the fundamental groups of closed hyperbolic
surfaces \cite{Dehn:1987}, continuing with small cancellation theory
\cite{Greendlinger:1960} and reaching more recent developments after
the advent of geometric group theory \cite{Gromov:1987}.  One such
development is the introduction of systolic complexes by Januszkiewicz
and {\' S}wi{\c a}tkowski \cite{Januszkiewicz:2006} and independently
by Haglund \cite{Haglund:2003}.  This class of simplicial complexes
first arose years earlier in the form of \defterm{bridged graphs}
defined by Soltan and Chepoi \cite{Soltan:1983, Chepoi:2000} in the
context of metric graph theory.  The flag completions of bridged
graphs are precisely the systolic complexes so these are essentially
the same objects.  The development of systolic complexes represents a
simplicial version of the cubical combinatorial nonpositive curvature
theory of $\CAT(0)$ cube complexes which were introduced by Gromov
\cite{Gromov:1987} but which can be traced back to median graphs
studied in metric graph theory \cite{Avann:1961, Nebesky:1971,
  Klavzar:1999, Roller:1998, Gerasimov:1998, Chepoi:2000}.  These two
theories have since been given a common generalization in the form of
the bucolic complexes \cite{Bresar:2013}.

In this paper we introduce the combinatorial nonpositive curvature
theory of quadric complexes.  Quadric complexes are closely related to
systolic complexes but have square rather than triangular $2$-cells, as
in the case of $\CAT(0)$ cube complexes.  We emphasize that although
$2$-dimensional $\CAT(0)$ cube complexes are quadric, the same does not
hold in higher dimensions where these theories differ strikingly.
Quadric complexes are defined by a disc diagrammatic nonpositive
curvature condition, similar to that described in Wise's presentation
of systolic complexes \cite{Wise:2003}.  Essentially, the local
condition satisfied by quadric complexes is that the ``star'' of any
``positively curved'' vertex of a disc diagram can be replaced by a
quadrangulation with no internal vertices while the local condition
satisfied by systolic complexes is that the star of any positively
curved vertex of a disc diagram can be replaced by a triangulation
with no internal vertices \cite{Wise:2003}.  Despite these
connections, there exist quadric groups that are not virtually
systolic (e.g. \Exref{elsner}) and quadric groups that are not
virtually cocompactly cubulated (e.g. \Exref{huang}).

As in the case of systolic complexes, the $1$-skeleta of quadric
complexes can be characterized by forbidden isometric subgraph
conditions.  Moreover, the cell structure of a quadric complex can be
recovered from its $1$-skeleton.  We thus find that the $1$-skeleta of
quadric complexes are precisely the hereditary modular graphs studied
in metric graph theory \cite{Bandelt:1988}.  Hence, as for systolic
complexes and $\CAT(0)$ cube complexes, a theory arising naturally in
geometric group theory has a precursor in metric graph theory.  By
some doubling, subdivision of squares and identification of cells,
quadric complexes can also be viewed as right-angled triangle
complexes.  They can thus be viewed as a generalization of the folder
complexes of Chepoi \cite{Chepoi:2000}, whose leg graphs satisfy our
forbidden subgraph conditions but which are further restricted.  Huang
and Osajda have introduced a common generalization of systolic and
quadric groups called metrically systolic groups, which are
essentially groups acting on metric simplicial complexes whose disc
diagrams are $\CAT(0)$ \cite{Huang:2017}.

Our main results are the following two theorems.

\begin{mainthm}[\Thmref{ibt}, Invariant Biclique Theorem]
  Let $G$ be a finite group acting on a locally finite quadric complex
  $X$, which is not equal to a single vertex.  Then $G$ stabilizes a
  biclique of $X$.
\end{mainthm}

In order to prove the Invariant Biclique Theorem finite quadric
complexes we use the fact, first proved by Bandelt
\cite[Theorem~2]{Bandelt:1988}, that the $1$-skeleta of finite quadric
complexes satisfy a form of dismantlability.  Our proof follows that
of Hensel et al. \cite{Hensel:2014} and Chepoi \cite{Chepoi:1997} for
finite systolic complexes.  We then apply a theorem of Hanlon and
Martinez-Pedroza \cite{Hanlon:2014} to lift this result to locally
finite quadric complexes.

The bi-dismantlability of $1$-skeleta of quadric complexes also plays an
essential role in an upcoming proof of the contractibility of quadric
complexes after the addition of certain higher dimensional cells
\cite{Hoda:2018}.

\begin{mainthm}[\Corref{cftfpresquad}]
  Let $G$ be a group admitting a finite $\cftf$ presentation
  $\gpres{\mcS}{\mcR}$.  Then $G$ acts properly and cocompactly on a
  quadric complex.
\end{mainthm}

We call a group acting properly and cocompactly on a quadric complex a
quadric group.  The proof that finitely presented $\cftf$ groups are
quadric uses the construction of a square complex $X_Y$ from a given
$2$-complex $Y$ with embedded $2$-cells.  We show that this square complex
$X_Y$ is simply connected when $Y$ is and that it is quadric when $Y$
is additionally $\cftf$.

\subsection{Structure of the Text}
\sseclabel{struct}

The rest of this section gives some basic definitions used throughout
the text and states conventions followed in the remaining sections.
\Secref{quadric} defines our main objects of study, quadric complexes
and quadric groups, and gives some of their basic properties.
\Secref{invbc} defines bi-dismantlability for bipartite graphs and uses
this property to prove the Invariant Biclique Theorem.  Finally,
\Secref{cftfgp} recalls the definition and basic properties of $\cftf$
complexes and proves that $\cftf$ groups are quadric.

\subsection{Basic Definitions}
\sseclabel{basicdef}

For fundamental notions such as that of \defterm{CW-complexes} and the
\defterm{fundamental group} see Hatcher's textbook on algebraic
topology \cite{Hatcher:2002}.  Let $X$ and $Y$ be $2$-dimensional
CW-complexes.  A \defterm{combinatorial map} from $X$ to $Y$ is a
continuous map whose restriction to every open cell $e$ of $X$ is a
homeomorphism from $e$ to an open cell of $Y$.  Two such maps are
considered the same if they are homotopic via a homotopy that is a
combinatorial map at each instant of time.  (Such a homotopy
necessarily restricts to an isotopy on each cell.)  A $2$-complex is
combinatorial if the attaching map of each of its $2$-cells is a
combinatorial map from the circle $\Sph^1$ endowed with the structure
of a $1$-dimensional CW-complex (i.e. a cycle graph).  A combinatorial
$2$-complex $X$ is \defterm{locally finite} if every cell of $X$
intersects finitely many other cells.

A \defterm{graph} is a $1$-dimensional CW complex $\Gamma$.  Every such
complex is combinatorial.  The \defterm{valence} of a $0$-cell of
$\Gamma$ is the number of ends of $1$-cells incident to it.  If no
$1$-cell of $\Gamma$ has both of its endpoints attached to the same
$0$-cell and no two $1$-cells of $\Gamma$ have their endpoints attached to
the same unordered pair of $0$-cells then $\Gamma$ is
\defterm{simplicial}.  The vertex set of any connected graph has a
natural metric, the \defterm{standard graph metric}, defined for a
pair of vertices $u$ and $v$ by the number of edges in the shortest
path connecting $u$ and $v$.  A simplicial graph $\Gamma$ is
\defterm{bipartite} if its $0$-cells can be partitioned into two
nonempty sets such that no $1$-cell has both of its endpoints in the
same part.  If every pair of $0$-cells from different parts is joined by
a $1$-cell then $\Gamma$ is a \defterm{complete bipartite graph} or a
\defterm{biclique}.  It is a fact that a simplicial graph is bipartite
if and only if it has no cycles of odd length, where a cycle is a
closed path.  A \defterm{square complex} is a combinatorial $2$-complex
whose $2$-cells are squares, that is its $2$-cell boundaries are endowed
with the structure of $4$-cycles.

A \defterm{disc diagram} $D$ is a compact contractible subspace of the
$2$-sphere $\Sph^2$ with the structure of a combinatorial $2$-complex.  A
disc diagram $D$ is \defterm{nonsingular} if it is homeomorphic to a
closed $2$-cell and is otherwise \defterm{singular}.  The topological
boundary of $D$ is denoted $\bd D$.  The boundary $\bd D$ is always a
subgraph of the $1$-skeleton $\sk{1}{D}$ of $D$.  A disc diagram
$D \subsetneq \Sph^2$ induces the structure of a combinatorial
$2$-complex on $\Sph^2$ with $D$ a subcomplex and $\Sph^2 \setminus D$
an open $2$-cell.  The attaching map $\Sph^1 \to \bd D$ of this $2$-cell
can be made combinatorial by pulling back the cell structure of
$\bd D$.  This turns $\Sph^1$ into a cycle denoted $\pbd D$.  The
resulting combinatorial map $\pbd D \to \bd D$ is the
\defterm{boundary path} of $D$.

Let $X$ be a combinatorial complex.  If $\mathcal{C}$ is a type of
combinatorial complex then a $\mathcal{C}$ \defterm{in} $X$ is a
$\mathcal{C}$ along with a combinatorial map from $C$ to $X$.  When
$D$ is a disc diagram in $X$, we abuse notation by also referring to
the concatenation $\pbd D \to \bd D \to X$ as the \defterm{boundary
  path} of $D$.

\subsection{Conventions Followed in the Text}
\sseclabel{convent}

We use the following conventions throughout the text unless otherwise
stated.  Maps and complexes are combinatorial.  Complexes are
connected.  Simply connected means connected and having trivial
fundamental group.  Distances between vertices in graphs are always
measured by the standard graph metric.  The notation $|\cdot,\cdot|$
is used to denote distance.  For graphs we use the terms vertex and
edge in place of $0$-cell and $1$-cell.  For square complexes we use the
terms vertex, edge and square.  For more general $2$-complexes we use
$0$-cell, $1$-cell and $2$-cell.

\subsection*{Acknowledgements}

The author would like to thank Daniel T. Wise for many invaluable
discussions during this research.  The author would also like to thank
a reviewer of this paper whose extensive comments led to many
corrections and improvements.

\section{Quadric Complexes}
\seclabel{quadric}

We now define our main object of study, quadric complexes.
\Ssecref{npcdiscs} describes locally minimal disc diagrams in quadric
complexes and shows that they are $\CAT(0)$ square complexes.  In
\Ssecref{catzdiscs} we recall properties of such disc diagrams that
are needed throughout the rest of the text.  \Ssecref{fbridged}
characterizes the $1$-skeleta of quadric complexes as those graphs whose
every isometrically embedded cycle is a $4$-cycle.  By a theorem of
Bandelt \cite{Bandelt:1988} these graphs are precisely those known as
\newterm{hereditary modular graphs} in the metric graph theory
literature.  Finally, in \Ssecref{hanped} we state a general theorem
of Hanlon and Martinez-Pedroza that implies that finitely presented
subgroups of quadric groups are quadric and state another theorem of
theirs needed in the proof of the Invariant Biclique Theorem.

\begin{figure}
  \centering
  \begin{tikzpicture}[scale=0.85]
    \begin{scope}
      \node[vertex,label={above:$u_0$}] (u0) at (1, 2) {};
      \node[vertex,label={left:$u_1$}] (u1) at (1, 4/3) {};
      \node[vertex,label={left:$u_2$}] (u2) at (1, 2/3) {};
      \node[vertex,label={below:$u_3$}] (u3) at (1, 0) {};
      \draw[thick] (u0) -- (u1);
      \draw[thick] (u1) -- (u2);
      \draw[thick] (u2) -- (u3);
      \draw[thick,postaction={decorate},decoration={markings,mark=at
      position 1/2 with {\arrow{>},\node[label={left:$e_1$}]
        {};}}] (u3) to[out=180,in=270] (0,1) to[out=90,in=180] (u0);
      \draw[thick,postaction={decorate},decoration={markings,mark=at
      position 1/2 with {\arrow{>},\node[label={right:$e_2$}]
        {};}}] (u3) to[out=0,in=270] (2,1) to[out=90,in=0] (u0);
    \end{scope}

    \draw (3,1) -- (3.75,1) node[above] {$f$};
    \draw[->] (3.75,1) -- (4.5,1);
    
    \begin{scope}[xshift=5cm]
      \node[vertex,label={above:$f(u_0)$}] (fu0) at (1, 2) {};
      \node[vertex,label={left:$f(u_1)$}] (fu1) at (1, 4/3) {};
      \node[vertex,label={left:$f(u_2)$}] (fu2) at (1, 2/3) {};
      \node[vertex,label={below:$f(u_3)$}] (fu3) at (1, 0) {};
      \draw[thick] (fu0) -- (fu1);
      \draw[thick] (fu1) -- (fu2);
      \draw[thick] (fu2) -- (fu3);
      \draw[thick,postaction={decorate},decoration={markings,mark=at
      position 1/2 with {\arrow{>},\node[label={right:$e = f(e_i)$}]
        {};}}] (fu3) to[out=0,in=270] (2,1) to[out=90,in=0] (fu0);
    \end{scope}
  \end{tikzpicture}
  \caption{The fold map $f$.  Let $g \colon s_1 \to s_2$ be an
    isomorphism of combinatorial complexes where $s_1$ and $s_2$ are
    squares.  Let $P_1 \subset \bd s_1$ be a combinatorial path of
    length $3$.  The domain of the fold map is
    $s_1 \sqcup s_2 /{\sim}$ where $x \sim g(x)$ for $x \in P_1$.  The
    fold map is the quotient of $s_1 \sqcup s_2 /{\sim}$ identifying
    $[x]_{\sim}$ and $[g(x)]_{\sim}$ for all $x \in s_1$.}
  \figlabel{foldmap}
\end{figure}

\begin{figure}
  \begin{subfigure}{\textwidth}
    \centering
    \begin{tikzpicture}[scale=0.85]
      \begin{scope}
        \node[vertex] (u0) at (1, 2) {};
        \node[vertex] (u1) at (1, 0) {};
        \node[vertex] (v0) at (0, 1) {};
        \node[vertex] (v1) at (1, 1) {};
        \node[vertex] (v2) at (2, 1) {};

        \draw[thick,postaction={decorate},decoration={markings,mark=at
          position 1/2 with {\arrow{>},\node[label={left:$a$}]
            {};}}] (v0) -- (u0);
        \draw[thick] (v1) -- (u0);
        \draw[thick,postaction={decorate},decoration={markings,mark=at
          position 1/2 with {\arrow{>},\node[label={right:$b$}]
            {};}}] (v2) -- (u0);
        \draw[thick,postaction={decorate},decoration={markings,mark=at
          position 1/2 with {\arrow{>},\node[label={left:$c$}]
            {};}}] (u1) -- (v0);
        \draw[thick] (u1) -- (v1);
        \draw[thick,postaction={decorate},decoration={markings,mark=at
          position 1/2 with {\arrow{>},\node[label={right:$d$}]
            {};}}] (u1) -- (v2);
      \end{scope}

      \draw[-implies, double equal sign distance] (3,1) -- (4,1);
      
      \begin{scope}[xshift=5cm]
        \node[vertex] (u0) at (1, 2) {};
        \node[vertex] (u1) at (1, 0) {};
        \node[vertex] (v0) at (0, 1) {};
        \node[vertex] (v2) at (2, 1) {};

        \draw[thick,postaction={decorate},decoration={markings,mark=at
          position 1/2 with {\arrow{>},\node[label={left:$a$}]
            {};}}] (v0) -- (u0);
        \draw[thick,postaction={decorate},decoration={markings,mark=at
          position 1/2 with {\arrow{>},\node[label={right:$b$}]
            {};}}] (v2) -- (u0);
        \draw[thick,postaction={decorate},decoration={markings,mark=at
          position 1/2 with {\arrow{>},\node[label={left:$c$}]
            {};}}] (u1) -- (v0);
        \draw[thick,postaction={decorate},decoration={markings,mark=at
          position 1/2 with {\arrow{>},\node[label={right:$d$}]
            {};}}] (u1) -- (v2);
      \end{scope}
    \end{tikzpicture}
    \caption{}
    \figlabel{repl2}
  \end{subfigure}
  \\
  \begin{subfigure}{\textwidth}
    \centering
    \begin{tikzpicture}[scale=0.85]
      \begin{scope}
        \node[vertex] (u) at (0, 0) {};
        \node[vertex] (b0) at (30:1) {};
        \node[vertex] (a0) at (90:1) {};
        \node[vertex] (b2) at (150:1) {};
        \node[vertex] (a2) at (210:1) {};
        \node[vertex] (b1) at (270:1) {};
        \node[vertex] (a1) at (330:1) {};

        \foreach \v in {a0,a1,a2}
          \draw[thick] (\v) -- (u);

        \draw[thick,postaction={decorate},decoration={markings,mark=at
          position 1/2 with {\arrow{>},\node[label={above:$c$}]
            {};}}] (a0) -- (b0);
        \draw[thick,postaction={decorate},decoration={markings,mark=at
          position 1/2 with {\arrow{>},\node[label={right:$d$}]
            {};}}] (b0) -- (a1);
        \draw[thick,postaction={decorate},decoration={markings,mark=at
          position 1/2 with {\arrow{>},\node[label={below:$e$}]
            {};}}] (a1) -- (b1);
        \draw[thick,postaction={decorate},decoration={markings,mark=at
          position 1/2 with {\arrow{>},\node[label={below:$f$}]
            {};}}] (b1) -- (a2);
        \draw[thick,postaction={decorate},decoration={markings,mark=at
          position 1/2 with {\arrow{>},\node[label={left:$a$}]
            {};}}] (a2) -- (b2);
        \draw[thick,postaction={decorate},decoration={markings,mark=at
          position 1/2 with {\arrow{>},\node[label={above:$b$}]
            {};}}] (b2) -- (a0);
      \end{scope}

      \draw[-implies,double equal sign distance] (1.75,0) -- (2.75,0);
      
      \begin{scope}[xshift=4.5cm]
        \node[vertex] (b0) at (30:1) {};
        \node[vertex] (a0) at (90:1) {};
        \node[vertex] (b2) at (150:1) {};
        \node[vertex] (a2) at (210:1) {};
        \node[vertex] (b1) at (270:1) {};
        \node[vertex] (a1) at (330:1) {};

        \draw[thick,postaction={decorate},decoration={markings,mark=at
          position 1/2 with {\arrow{>},\node[label={above:$c$}]
            {};}}] (a0) -- (b0);
        \draw[thick,postaction={decorate},decoration={markings,mark=at
          position 1/2 with {\arrow{>},\node[label={right:$d$}]
            {};}}] (b0) -- (a1);
        \draw[thick,postaction={decorate},decoration={markings,mark=at
          position 1/2 with {\arrow{>},\node[label={below:$e$}]
            {};}}] (a1) -- (b1);
        \draw[thick,postaction={decorate},decoration={markings,mark=at
          position 1/2 with {\arrow{>},\node[label={below:$f$}]
            {};}}] (b1) -- (a2);
        \draw[thick,postaction={decorate},decoration={markings,mark=at
          position 1/2 with {\arrow{>},\node[label={left:$a$}]
            {};}}] (a2) -- (b2);
        \draw[thick,postaction={decorate},decoration={markings,mark=at
          position 1/2 with {\arrow{>},\node[label={above:$b$}]
            {};}}] (b2) -- (a0);

        \draw[thick] (b2) -- (a1);
      \end{scope}

      \node at (6.25,0) {,};

      \begin{scope}[xshift=8cm]
        \node[vertex] (b0) at (30:1) {};
        \node[vertex] (a0) at (90:1) {};
        \node[vertex] (b2) at (150:1) {};
        \node[vertex] (a2) at (210:1) {};
        \node[vertex] (b1) at (270:1) {};
        \node[vertex] (a1) at (330:1) {};

        \draw[thick,postaction={decorate},decoration={markings,mark=at
          position 1/2 with {\arrow{>},\node[label={above:$c$}]
            {};}}] (a0) -- (b0);
        \draw[thick,postaction={decorate},decoration={markings,mark=at
          position 1/2 with {\arrow{>},\node[label={right:$d$}]
            {};}}] (b0) -- (a1);
        \draw[thick,postaction={decorate},decoration={markings,mark=at
          position 1/2 with {\arrow{>},\node[label={below:$e$}]
            {};}}] (a1) -- (b1);
        \draw[thick,postaction={decorate},decoration={markings,mark=at
          position 1/2 with {\arrow{>},\node[label={below:$f$}]
            {};}}] (b1) -- (a2);
        \draw[thick,postaction={decorate},decoration={markings,mark=at
          position 1/2 with {\arrow{>},\node[label={left:$a$}]
            {};}}] (a2) -- (b2);
        \draw[thick,postaction={decorate},decoration={markings,mark=at
          position 1/2 with {\arrow{>},\node[label={above:$b$}]
            {};}}] (b2) -- (a0);

        \draw[thick] (a0) -- (b1);
      \end{scope}

      \node at (9.75,0) {,};

      \begin{scope}[xshift=11.5cm]
        \node[vertex] (b0) at (30:1) {};
        \node[vertex] (a0) at (90:1) {};
        \node[vertex] (b2) at (150:1) {};
        \node[vertex] (a2) at (210:1) {};
        \node[vertex] (b1) at (270:1) {};
        \node[vertex] (a1) at (330:1) {};

        \draw[thick,postaction={decorate},decoration={markings,mark=at
          position 1/2 with {\arrow{>},\node[label={above:$c$}]
            {};}}] (a0) -- (b0);
        \draw[thick,postaction={decorate},decoration={markings,mark=at
          position 1/2 with {\arrow{>},\node[label={right:$d$}]
            {};}}] (b0) -- (a1);
        \draw[thick,postaction={decorate},decoration={markings,mark=at
          position 1/2 with {\arrow{>},\node[label={below:$e$}]
            {};}}] (a1) -- (b1);
        \draw[thick,postaction={decorate},decoration={markings,mark=at
          position 1/2 with {\arrow{>},\node[label={below:$f$}]
            {};}}] (b1) -- (a2);
        \draw[thick,postaction={decorate},decoration={markings,mark=at
          position 1/2 with {\arrow{>},\node[label={left:$a$}]
            {};}}] (a2) -- (b2);
        \draw[thick,postaction={decorate},decoration={markings,mark=at
          position 1/2 with {\arrow{>},\node[label={above:$b$}]
            {};}}] (b2) -- (a0);

        \draw[thick] (a2) -- (b0);
      \end{scope}
    \end{tikzpicture}
    \caption{}
    \figlabel{repl3}
  \end{subfigure}
  \caption{Replacement rules for quadric complexes.}
  \figlabel{repl}
\end{figure}

\begin{defn}
  \defnlabel{quadric} A \defterm{locally quadric complex} is a square
  complex $X$ satisfying the following conditions.
  \begin{enumerate}
  \item \itmlabel{imm} The attaching map of every square is an immersion.
  \item \itmlabel{piec3} Any disc diagram in $X$ of the form of the
    domain of the fold map factors through the fold map.  The
    \defterm{fold map} is described in \Figref{foldmap}.
  \item \itmlabel{rep2} For any disc diagram in $X$ of the form of the
    left-hand side of \Figref{repl2} with immersed boundary, there is
    a disc diagram in $X$ of the form on the right with the same
    boundary path.
  \item \itmlabel{rep3} For any disc diagram in $X$ of the form of the
    left-hand side of \Figref{repl3} with immersed boundary, there is
    a disc diagram in $X$ of one of the forms on the right with the
    same boundary path.
  \end{enumerate}
  A \defterm{quadric complex} is a simply connected locally quadric
  complex.
\end{defn}

Condition~\defnitmref{quadric}{piec3} implies that no two squares have
the same attaching map.  Conditions \defnitmref{quadric}{rep2} and
\defnitmref{quadric}{rep3} are nonpositive curvature requirements
having important consequences for disc diagrams in locally quadric
complexes.

Quadric complexes are similar in nature to systolic complexes.  This
is especially apparent in the presentation given by Wise
\cite{Wise:2003}.  Wise also introduces ``generalized
$(p,q)$-complexes'' which encompass systolic complexes as a subclass
of generalized $(3,6)$-complexes and quadric complexes as a subclass
of generalized $(4,4)$-complexes \cite{Wise:2003}.

The following proposition follows immediately from \Defnref{quadric}.
\begin{prop}
  The class of locally quadric complexes is closed under the
  operations of taking full subcomplexes and taking covering spaces.
\end{prop}
A \newterm{full} subcomplex is one that includes any cell whose
boundary is in the subcomplex.

\begin{defn*}
  A group is \defterm{quadric} if it acts properly and cocompactly on
  a quadric complex.
\end{defn*}

If $X$ is a compact, connected locally quadric complex, then its
universal cover $\ucov{X}$ is quadric and so its fundamental group
$\pi_1(X)$ is quadric.

\subsection{Examples}
\seclabel{examples}

We now discuss a few classes of examples of quadric complexes and
groups.  It follows immediately from \Defnref{quadric} that $\CAT(0)$
square complexes are quadric.
\begin{ex}\exlabel{elsner}
  The following example of Elsner and Przytycki \cite{Elsner:2013} is
  of a nonpositively curved square complex whose fundamental group
  does not virtually act properly and cocompactly on a systolic
  complex \cite[Theorem~4.1]{Elsner:2013}.
  \[\gpres{a,b,c}{aba^{-1}b,c^{-1}ac=b}\]
  Hence, this group is quadric but not virtually systolic.  This shows
  that though the quadric and systolic theories share many
  similarities, they are nevertheless distinct.
\end{ex}

\begin{figure}
  \centering
  \begin{tikzpicture}[scale=1.5]
    \begin{scope}
      \node[vertex] (u0) at (1, 2) {};
      \node[vertex] (u1) at (1, 0) {};
      \node[vertex] (v0) at (0, 1) {};
      \node[vertex] (v1) at (1, 1) {};
      \node[vertex] (v2) at (2, 1) {};

      \draw[thick,postaction={decorate},decoration={markings,mark=at
        position 1/2 with {\arrow{>},\node[label={left:$\sigma(j)$}]
          {};}}] (v0) to[out=60,in=210] (u0);
      \draw[thick,postaction={decorate},decoration={markings,mark=at
        position 1/3 with {\arrow{>},\node[label={right:$\sigma(i)$}]
          {};}}] (v1) -- (u0);
      \draw[thick,postaction={decorate},decoration={markings,mark=at
        position 1/2 with {\arrow{>},\node[label={right:$\sigma(k)$}]
          {};}}] (v2) to[out=120,in=330] (u0);
      \draw[thick,postaction={decorate},decoration={markings,mark=at
        position 1/2 with {\arrow{>},\node[label={left:$j$}]
          {};}}] (u1) to[out=150,in=300] (v0);
      \draw[thick,postaction={decorate},decoration={markings,mark=at
        position 2/3 with {\arrow{>},\node[label={right:$i$}]
          {};}}] (u1) -- (v1);
      \draw[thick,postaction={decorate},decoration={markings,mark=at
        position 1/2 with {\arrow{>},\node[label={right:$k$}]
          {};}}] (u1) to[out=30,in=240] (v2);
    \end{scope}
  \end{tikzpicture}
  \caption{A disc diagram in the presentation of \Exref{permpres}.}
  \figlabel{exfig}
\end{figure}

\begin{ex}\exlabel{permpres}
  Let $n$ be a positive integer.  let $\sigma$ be a permutation of
  $\{1, 2, \ldots, n\}$.  The presentation
  \[ \gpres{g_1, g_2, \ldots,
      g_n}{\text{$g_ig_{\sigma(i)}=g_jg_{\sigma(j)}$, for all
        $i \neq j$}} \] is locally quadric and so presents a quadric
  group.  Indeed, no disc diagram of the form on the left-hand side of
  \Figref{repl3} has immersed boundary and any disc diagram of the
  form on the left-hand side of \Figref{repl2} with immersed boundary
  is as in \Figref{exfig} and so has boundary which bounds a square.
  Setting $n=3$ and letting $\sigma$ be the permutation $(1\; 2\; 3)$
  we obtain a presentation of the braid group $B_3$ on three strands.
\end{ex}

\begin{ex}\exlabel{huang}
  By \Corref{cftfpresquad}, $\cftf$ small cancellation groups are
  quadric.  In particular, the Artin group
  \[\gpres{a,b,c}{ab=ba, bcb=cbc}\]
  is quadric.  However, by Huang, Jankiewicz and Przytycki
  \cite[Theorem~1.2]{Huang:2016} or, independently, Haettel
  \cite[Theorem~A]{Haettel:2015}, this group does not virtually act
  properly and cocompactly on a $\CAT(0)$ cube complex.  In general,
  Artin groups whose defining graphs are triangle free are $\cftf$
  (see Pride \cite{Pride:1986}) and many such Artin groups are not
  virtually cocompactly cubulated \cite{Haettel:2015,
    Huang:2016}.\footnote{Thanks to Jingyin Huang for bringing this
    class of examples to the author's attention.}
\end{ex}

\subsection{Nonpositive Curvature of Disc Diagrams}
\sseclabel{npcdiscs}

The nonpositive curvature conditions \defnitmref{quadric}{rep2} and
\defnitmref{quadric}{rep3} in \Defnref{quadric} imply that any
nullhomotopic closed path in a locally quadric complex bounds a
nonpositively curved disc diagram.

A disc diagram $D$ in a $2$-complex $X$ has \defterm{minimal area} if it
contains the minimal number of $2$-cells over all disc diagrams in $X$
having the same boundary path.  A disc diagram $D$ in a quadric
complex is \defterm{locally minimal} if every internal vertex of $D$
is incident to at least four squares.

\begin{lem}
  \lemlabel{minardd} Let $X$ be a locally quadric complex and $D$ a
  minimal area disc diagram in $X$.  Then $D$ is locally minimal.
\end{lem}
\begin{proof}
  Suppose $v$ is an internal vertex of $D$ incident to $k < 4$
  squares.  By Condition~\defnitmref{quadric}{imm}, $k \ge 2$.  Let
  $D'$ be the union of closed $2$-cells incident to $v$.
  
  Suppose first that $k=2$.  Then $D'$ either has the form of the disc
  diagram in Condition~\defnitmref{quadric}{piec3} or that of the disc
  diagram in Condition~\defnitmref{quadric}{rep2}.  In the former case
  the exterior edges of the two squares in $D'$ map to the same edge
  of $X$ and with the same orientation so that $D$ can be replaced by
  a disc diagram obtained by cutting out $D'$ and gluing together
  these exterior edges.  In the latter case $D'$ can be replaced by a
  single square if its boundary immerses.  If the boundary of $D'$
  does not immerse then we may cut $D$ along a nonimmersing path of
  length $2$ in $\bd D'$ and glue it back together in such a way as to
  introduce a subdisc of the form in
  Condition~\defnitmref{quadric}{piec3}.  In any case we contradict
  the minimality of the area of $D$.
  
  Suppose now that $k=3$.  By the $k=2$ case we may assume that every
  vertex of $D$ is incident to at least three squares.  Then $D'$ has
  the form of the disc diagram in
  Condition~\defnitmref{quadric}{rep3}.  If the $\bd D'$ immerses then
  $D'$ can be replaced by a pair of squares.  If $\bd D'$ does not
  immerse then we may cut $D$ along a nonimmersing path of length $2$
  in $\bd D'$ and glue it back together in such a way as to introduce
  a vertex incident to two squares.  In any case we again contradict
  the minimality of the area of $D$.
\end{proof}

We see from the proof of \Lemref{minardd} that, given a disc diagram
$D$ in a locally quadric complex, we can obtain a locally minimal disc
diagram with the same boundary path by performing a finite number of
replacements.  Each replacement reduces the number of squares, though
the locally minimal disc diagram we ultimately obtain may not be of
minimal area.

\subsubsection{\texorpdfstring{$\CAT(0)$}{CAT(0)} Disc Diagrams}
\ssseclabel{catzdiscs}

A \defterm{$\CAT(0)$ cube complex} is a cube complex for which the
metric obtained by making each cube isometric to a standard Euclidean
cube satisfies a metric nonpositive curvature condition concerning the
thinness of its triangles.  Such complexes have been studied
extensively in the geometric group theory literature
\cite{Sageev:1995, Wise:2011, Wise:2012, Niblo:1998, Haglund:2008,
  Niblo:1997} as well as in the metric graph theory literature, via
their $1$-skeleta \cite{Roller:1998, Gerasimov:1998, Chepoi:2000},
\newterm{median graphs} \cite{Avann:1961, Nebesky:1971, Klavzar:1999}.

In this paper we make use of a purely combinatorial characterization
of the $\CAT(0)$ condition for cube complexes which is due to Gromov
\cite{Gromov:1987}.  In dimension $2$ this characterization, which we
may take as a definition, is as follows.  A square complex $X$ is
\defterm{$\CAT(0)$} if and only if $X$ is simply connected and the
shortest embedded cycle in the link of any $0$-cell of $X$ has length
at least $4$.  The \defterm{link} of a $0$-cell $v$ of $X$ is the
graph whose vertices correspond to ends of $1$-cells of $X$ incident
to $v$ and whose edges correspond to corners of $2$-cells of $X$
incident to $v$.  In the case where $X$ is a disc diagram, the
$\CAT(0)$ condition is equivalent to the condition that each interior
vertex of $X$ is incident to at least four squares.

By their definition, locally minimal disc diagrams in locally quadric
complexes are \newterm{$\CAT(0)$ square complexes}.  In particular, by
\Lemref{minardd}, minimal area disc diagrams in locally quadric
complexes are $\CAT(0)$.  We refer to disc diagrams that are $\CAT(0)$
square complexes as \defterm{$\CAT(0)$ disc diagrams} for brevity.
Such disc diagrams are amenable to standard arguments using a
combinatorial version of the Gauss-Bonnet Theorem as well as to
well-known dual curve constructions.  These and other well-known
properties of $\CAT(0)$ disc diagrams will be recalled in the rest of
this section.  Most of these results can be found in Wise's study of
minimal area cubical disc diagrams \cite{Wise:2011}.

Let $v$ be a vertex of a disc diagram $D$ with square $2$-cells.  Let
$\delta(v)$ denote the valence of $v$ and $\rho(v)$ the number of
corners of squares incident to $v$.  The \defterm{curvature} of $v$ is
\[ \kappa(v) = 2\pi - \delta(v)\pi + \frac{\rho(v)}{2}\pi. \] This
notion of curvature may be considered as emerging from the metric
obtained on $D$ by treating each square as a standard Euclidean
square.  Intuitively, this forces all of the curvature of $D$ to be
concentrated at its vertices.

\begin{lem}
  \lemlabel{pospos} Let $D$ be a disc diagram with square $2$-cells.
  Assume that $D$ is not a single vertex and let $v$ be a vertex of
  $\bd D$.  If $\kappa(v) > 0$ then $v$ is not a cutpoint of $D$ and
  $\kappa(v)$ is either $\pi$ or $\frac{\pi}{2}$.  If
  $\kappa(v) = \pi$ then $v$ has valence $1$ and is not incident to
  any squares.  If $\kappa(v) = \frac{\pi}{2}$ then $v$ has valence
  $2$ and is incident to a single corner of a square.  If
  $\kappa(v) = 0$ then either $v$ is a cutpoint, has valence $2$ and
  is not incident to any corners of squares or $v$ is not a cutpoint,
  has valence $3$ and is incident to $2$ corners of squares.
\end{lem}
\begin{proof}
  We first consider the case where $v$ is not a cut point of $D$.  If
  $v$ is not a cutpoint then $\delta(v) = \rho(v) + 1$ and so
  \[ \kappa(v) = 2\pi - \bigl(\rho(v) + 1\bigr)\pi +
    \frac{\rho(v)}{2}\pi = \pi - \frac{\rho(v)}{2}\pi \] which can
  only take the nonnegative values $\pi$ and $\frac{\pi}{2}$ and $0$.
  If $\kappa(v) = \pi$ then $\rho(v) = 0$ and so $\delta(v) = 1$.  If
  $\kappa(v) = \frac{\pi}{2}$ then $\rho(v) = 1$ and so
  $\delta(v) = 2$.  If $\kappa(v) = 0$ then $\rho(v) = 2$ and so
  $\delta(v) = 3$.
  
  If $v$ is a cut point of $D$ then consider the disc diagrams
  $(D_i)_{i=1}^n$ obtained as closures of the components of
  $D \setminus \{v\}$.  Then, by the preceding paragraph, we have
  \[ \kappa(v) = 2\pi + \sum_{i=1}^n\bigl(\kappa_{D_i}(v) - 2\pi\bigr)
    \le 2\pi - n\pi = (2-n)\pi \le 0 \] where $\kappa_{D_i}(v)$ is the
  curvature of $v$ in $D_i$.  We have equality if and only if
  $\kappa_{D_i}(v) = \pi$, for each $i$, and $n = 2$.  So, by the
  preceding paragraph, we have equality if and only if $v$ is not
  incident to any corners of squares and has valence $2$ in $D$.
\end{proof}

\begin{prop}[Gauss-Bonnet Theorem for $\CAT(0)$ Disc Diagrams]
  \proplabel{gbcat} Let $D$ be a disc diagram with square $2$-cells and
  assume that $D$ is not a single vertex.  The sum of the curvatures
  of vertices of $D$ is $2\pi$, i.e.,
  \[ \sum_{\text{$v$ vertex}} \kappa(v) = 2\pi.\]
\end{prop}
\begin{proof}
  The Euler characteristic $\chi(D)$ of $D$ can be computed by
  subtracting the number of its edges from the number of its vertices
  and squares.  That is, each edge contributes $-1$ to the Euler
  characteristic and each vertex or square contributes $1$.  Evenly
  distributing the $-1$ of each edge to its vertices and the $1$ from
  each square to the vertices on its boundary gives the sum
  \[ \sum_{\text{$v$ vertex}} \frac{\kappa(v)}{2\pi}, \]
  which then must equal $\chi(D) = 1$.
\end{proof}

\begin{cor}[Greendlinger's Lemma for $\CAT(0)$ Disc Diagrams]
  \corlabel{greendlingercat} Let $D$ be a $\CAT(0)$ disc diagram and
  assume that $D$ is not a single vertex.  Then there are at least two
  positively curved vertices on the boundary $\bd D$ of $D$.  If $D$
  has no valence $1$ vertices then there are at least four vertices on
  $\bd D$ with curvature $\frac{\pi}{2}$.
\end{cor}
\begin{proof}
  The $\CAT(0)$ property implies that no interior vertex of $D$ is
  positively curved.  Then, by \Propref{gbcat}, there must be at least
  $2\pi$ positive curvature on $\bd D$.  By \Lemref{pospos}, the
  curvature of vertices is bounded above by $\pi$ and so there must be
  at least two positively curved vertices of $\bd D$.
  
  If $D$ has no valence $1$ vertices then, by \Lemref{pospos}, the
  curvature of its vertices is bounded above by $\frac{\pi}{2}$ and
  this is the least positive curvature of a vertex of $D$.  Hence
  $\bd D$ must have at least four vertices of $\frac{\pi}{2}$
  curvature.
\end{proof}

\begin{defn*}
  Let $D$ be a disc diagram with square $2$-cells.  The
  \defterm{midcube} of an edge $e$ of $D$ is the midpoint of $e$.  The
  \defterm{midcube} of a square of $D$ is the closed line segment
  joining the midcubes of a pair of its opposing edges.  A
  \defterm{dual curve} $\alpha$ of $D$ is a minimal subspace of $D$
  satisfying the following conditions.
  \begin{enumerate}
  \item Some midcube of $D$ is contained in $\alpha$.
  \item If $\alpha$ intersects a midcube $\mu$ of $D$ nontrivially
    then it contains $\mu$.
  \end{enumerate}
\end{defn*}

The squares of $D$ are in one-to-one correspondence with the
intersections of its dual curves.

Let $D$ be a disc diagram with square $2$-cells.  A \defterm{nonogon} of
$D$ is a dual curve of $D$ homeomorphic to the circle $S^1$.  A
self-intersecting dual curve of $D$ forms a \defterm{monogon} of $D$.
A pair of dual curves of $D$ that intersect twice form a
\defterm{bigon} of $D$.  Finally, three pairwise intersecting dual
curves form a \defterm{triangle} of $D$.

\begin{prop}
  \proplabel{nonogon} If $D$ is a $\CAT(0)$ disc diagram then it has
  no nonogons, monogons, bigons or triangles.
\end{prop}
\begin{proof}
  If $D$ has a nonogon, monogon, bigon or triangle, then it has a dual
  curve polygon with $k = 0$, $1$, $2$, or $3$ corners.  Let $N$ be
  the union of the open cells intersecting the boundary of this
  polygon.  Let $D'$ be the interior component of $D \setminus N$ and
  let $D''$ be the closure of $D' \cup N$.  Note that $D'$ is not a
  single vertex $v$ since then $v$ would be incident to exactly $k$
  squares and this would contradict the $\CAT(0)$ condition.

  For $v \in \bd D'$ let $\rho_N(v)$ be the number of open squares of
  $N$ whose closures contain $v$, let $\rho'(v)$ be the number of
  squares of $D'$ incident to $v$, let $\delta'(v)$ be the valence of
  $v$ in $D'$ and let $\kappa'(v)$ be the curvature of $v$ in $D'$.
  By the $\CAT(0)$ property we have $\rho_N(v) + \rho'(v) \ge 4$ and
  so
  \[ \rho_N(v) - 2 \ge 2 - \rho'(v) = 4 - 2\bigl(\rho'(v) + 1\bigr) +
    \rho'(v) \ge 4 - 2\delta'(v) + \rho'(v) =
    \frac{2}{\pi}\kappa'(v) \] where the final inequality relies on
  the fact that $v \in \bd D'$.  But notice that
  \[ k = \sum_{v \in \bd D'} \bigl(\rho_N(v) - 2 \bigr) \] while the proof
  of \Corref{greendlingercat} tells us that
  \[ \sum_{v \in \bd D'} \kappa'(v) \ge 2\pi \] and so we have the
  following contradiction.
  \[ k = \sum_{v \in \bd D'} \bigl(\rho_N(v) - 2 \bigr) \ge
    \frac{2}{\pi}\sum_{v \in \bd D'} \kappa'(v) \ge 4 \]
\end{proof}

Note that the ``no nonogons'' part of \Propref{nonogon} implies that
if $D$ is $\CAT(0)$ and has boundary path length $|\pbd D| = n$ then
it has exactly $\frac{n}{2}$ dual curves.

The technique of eliminating $n$-gons of dual curves for low $n$ of
disc diagrams in square complexes comes from unpublished lecture notes
of Casson and has been developed in the context of $\CAT(0)$ cube
complexes \cite{Sageev:1995, Wise:2012}.

The absence of triangles in $\CAT(0)$ disc diagrams distinguishes them
from general disc diagrams in $\CAT(0)$ cube complexes of dimension at
least $3$.  Such disc diagrams may have triangles, as they contain
internal vertices of valence $3$, though they do not have nonogons,
monogons or bigons.

\begin{lem}
  \lemlabel{protorail} Let $D$ be a disc diagram with square
  $2$-cells. Let $\zeta$ be a subpath of nonzero length of $\bd
  D$. Suppose the initial and terminal vertices of $\zeta$ each have
  curvature $\frac{\pi}{2}$ and that each internal vertex of $\zeta$
  has curvature $0$.  Then the inclusion of $\zeta$ extends to a map
  $f \colon [0,1] \times [0,n] \to D$ from a $1 \times n$ grid of
  squares satisfying the following conditions.
  \begin{enumerate}
  \item The restriction of $f$ to $\{1\} \times [0,n]$ is an
    isomorphism onto $\zeta$.
  \item $f$ embeds
    $[0,1] \times \{0\} \cup \{1\} \times [0,n] \cup [0,1] \times
    \{n\}$ into $\bd D$.
  \item The restriction of $f$ to
    $\bigl\{\frac{1}{2}\bigr\} \times [0,n]$ maps onto a dual curve of
    $D$.
  \item $f$ restricts to an embedding on $(0,1] \times [0,n]$.
  \end{enumerate}
\end{lem}
\begin{proof}
  By \Lemref{pospos}, the initial and terminal vertices are not
  cutpoints and are each incident to a square.  By \Lemref{pospos},
  for each internal vertex $v$ of $\zeta$, either $v$ is a cutpoint
  and $v$ is not incident to any square or $v$ is not a cutpoint and
  $v$ is incident to two squares.  Assume for the sake of
  contradiction that $v$ is the first vertex of $\zeta$ which is a
  cutpoint of $D$.  Then the preceding vertex $v'$ is not a cutpoint
  so is incident to a square.  Hence $v$ must also be incident to a
  square and this is a contradiction.  So no vertex of $\zeta$ is a
  cutpoint and so, by \Lemref{pospos}, each internal vertex of $\zeta$
  is incident to two corners of squares and has valence $3$.  It
  follows that the inclusion of $\zeta$ extends to a map
  $f \colon [0,1] \times [0,n] \to D$ with $\{1\} \times [0,n]$
  mapping onto $\zeta$ by an isomorphism and with
  $[0,1] \times \{0\} \cup \{1\} \times [0,n] \cup [0,1] \times \{n\}$
  embedding in $\bd D$.  The map $f$ is also an immersion along each
  edge $[0,1] \times \{k\}$ and so the dual curve
  $\bigl\{\frac{1}{2}\bigr\} \times [0,n]$ maps onto a dual curve of
  $D$.  Then, by \Propref{nonogon}, the map $f$ is injective on
  $(0,1] \times [0,n]$.
\end{proof}

\begin{figure}
  \centering
  \begin{tikzpicture}
    \begin{scope}[scale=4/3]
      \pgfmathsetmacro\inradius{1}
      \pgfmathsetmacro\thickness{0.5}
      \pgfmathsetmacro\midradius{\inradius+\thickness/2}
      \pgfmathsetmacro\outradius{\inradius+\thickness}
      \pgfmathsetmacro\numsegments{7}

      \coordinate (e10) at (90:\outradius);
      \coordinate (e1m) at (90:\midradius);
      \coordinate (e11) at (90:\inradius);
      \coordinate (e20) at (-90:\inradius);
      \coordinate (e2m) at (-90:\midradius);
      \coordinate (e21) at (-90:\outradius);
      
      \node[vertex] at (e10) {};
      \node[vertex] at (e11) {};
      \node[vertex] at (e20) {};
      \node[vertex] at (e21) {};

      \draw[thick] (e11) arc (90:-90:\inradius);
      \draw[thick,dotted] (e1m) arc (90:-90:\midradius);
      \draw[thick,postaction={decorate},decoration={markings,mark=at
        position 1/2 with {\arrow{>},\node[label={right:$\beta$}]
          {};}}] (e10) arc (90:-90:\outradius);

      \draw[thick,postaction={decorate},decoration={markings,mark=at
        position 1/2 with {\arrow{>},\node[label={left:$e_1$}]
          {};}}] (e10) -- (e11);
      \draw[thick,postaction={decorate},decoration={markings,mark=at
        position 1/2 with {\arrow{>},\node[label={left:$e_2$}]
          {};}}] (e20) -- (e21);

      \draw[thick,postaction={decorate},decoration={markings,mark=at
        position 1/2 with {\arrow{>},\node[label={left:$\gamma'$}]
          {};}}] (e11) -- (e20);

      \pgfmathsetmacro\nsmo{\numsegments - 1}
      \foreach \i in {1,...,\nsmo}
      {
        \pgfmathsetmacro\angle{-90 + \i * 180 / \numsegments}
        \draw[thick] (\angle:\inradius) -- (\angle:\outradius);
        \node[vertex] at (\angle:\inradius) {};
        \node[vertex] at (\angle:\outradius) {};
      }
    \end{scope}
  \end{tikzpicture}
  \caption[A $\CAT(0)$ disc diagram from the proof of
  \Lemref{geodcurve}.]{A path $e_1 \gamma' e_2$ twice crossing a dual
    curve $\alpha$ of a $\CAT(0)$ disc diagram, as in the proof of
    \Lemref{geodcurve}.  The dual curve $\alpha$ is shown as a dotted
    line.}
  \figlabel{crossdctwice}
\end{figure}

\begin{lem}
  \lemlabel{geodcurve} Let $D$ be a $\CAT(0)$ disc diagram and let
  $\gamma$ be a geodesic of its $1$-skeleton $\sk{1}{D}$.  Then $\gamma$
  does not cross any dual curve of $D$ more than once.
\end{lem}
\begin{proof}
  Suppose $\gamma$ crosses the dual curve $\alpha$ at least twice.
  Let $e_1$ and $e_2$ be two directed edges along which $\gamma$
  crosses $\alpha$ such that the subpath $\gamma'$ of $\gamma$ from
  the terminal point of $e_1$ to the initial point of $e_2$ does not
  cross $\alpha$.  Let $\beta$ be the path of $\sk{1}{D}$ from the
  initial point of $e_1$ to the terminal point of $e_2$ running
  parallel to $\alpha$.  Let $D'$ be the subdisc of $D$ with boundary
  $\bd D' = e_1 \gamma' e_2 \beta^{-1}$, as in \Figref{crossdctwice}.
  
  Now, suppose we chose $\gamma$, $\alpha$ and the $e_i$ so as to
  minimize the area of $D$.  Then no dual curve starts and ends on
  $\gamma'$.  Hence the dual curves starting on $\gamma'$ end on
  $\beta$. The subdisc $D'$ is $\CAT(0)$ so, by the prohibition in
  \Propref{nonogon} of bigons, every dual curve starting on $\beta$
  must end on $\gamma'$.  So the dual curves of $\gamma'$ and $\beta$
  are one and the same.  Then $\gamma'$ and $\beta$ have the same
  length and so $\beta$ is a shorter path than $e_1 \gamma' e_2$.
  Then $\gamma$ is not a geodesic---a contradiction.
\end{proof}

\begin{prop}
  \proplabel{carrier} Let $\alpha$ be a dual curve of a $\CAT(0)$ disc
  diagram $D$.  Let $[0,1] \times [0,n]$ be a $1 \times n$ grid of
  squares with $n = |\alpha|$.  Let
  $f \colon [0,1] \times [0,n] \to D$ be the map whose restriction to
  $\{\frac{1}{2}\} \times [0,n]$ is the inclusion of $\alpha$.  Then
  $f$ is an embedding and the $1$-skeleton of $[0,1] \times [0,n]$ is
  convex in the $1$-skeleton $\sk{1}{D}$ of $D$.  We call the image of
  $f$ the \defterm{carrier} of $\alpha$.
\end{prop}
\begin{proof}
  By \Propref{nonogon}, the restriction of $f$ to $(0,1) \times [0,n]$
  is an embedding and the complement of $\alpha$ in $D$ has two
  components.  So, it suffices to show that the restrictions of $f$ to
  the parallel paths $\{0\} \times [0,n]$ and $\{1\} \times [0,n]$ are
  convex embeddings.  Suppose there is a geodesic
  $\gamma \colon P \to D$ from some $f(0,k)$ to some $f(0,k')$, with
  $k \le k'$, that does not coincide with $f|_{\{0\} \times [k,k']}$.
  Choose such a $\gamma$ minimizing $(|\gamma|,k'-k)$
  lexicographically.  By minimality of $\gamma$, the closed path given
  by concatenating $\gamma$ and $f|_{\{0\} \times [k,k']}$ is
  embedded.  Hence the $\CAT(0)$ subdisc diagram $D' \subset D$
  bounded by the concatenation of $\gamma$ and
  $f|_{\{0\} \times [k,k']}$ contains at least one square $s$.  By
  \Propref{nonogon}, two distinct dual curves $\alpha'$ and $\alpha''$
  cross at $s$.  By \Lemref{geodcurve} and \Propref{nonogon}, no dual
  curve of $D'$ has both ends on $P$ or both ends on
  $\{0\} \times [k,k']$.  Hence $\alpha$, $\alpha'$ and $\alpha''$
  pairwise intersect, contradicting \Propref{nonogon}.
\end{proof}

\begin{prop}
  \proplabel{dcdistance} Let $D$ be a $\CAT(0)$ disc diagram and let
  $u$ and $v$ be vertices of $D$.  Then the distance between $u$ and
  $v$ in the $1$-skeleton $\sk{1}{D}$ if $D$ is equal to the number $n$
  of dual curves of $D$ separating $u$ from $v$.
\end{prop}
\begin{proof}
  Let $\gamma$ be a geodesic of $\sk{1}{D}$ from $u$ to $v$.  Then
  $\gamma$ must cross the dual curves separating $u$ from $v$ and so
  must traverse at least $n$ edges.  Suppose $\gamma$ traversed some
  additional edge $e$ and let $\alpha$ be the dual curve containing
  the midcube of $e$.  By \Lemref{geodcurve}, $\alpha$ is not a dual
  curve separating $u$ and $v$.  But then $\alpha$ must be traversed a
  second time if $\gamma$ is to end at $v$.  This is impossible, by
  \Lemref{geodcurve} and so $\gamma$ traverses exactly $n$ edges and
  the distance $|u,v|$ between $u$ and $v$ in $\sk{1}{D}$ is $n$.
\end{proof}

\subsection{Quadratic Isoperimetric Function and Word Problem}
\sseclabel{quadisoper}

The $\CAT(0)$ property of disc diagrams in quadric complexes implies
the following proposition and corollaries.

\begin{prop}
  Let $X$ be a quadric complex.  Then $X$ has quadratic isoperimetric
  function.
\end{prop}
\begin{proof}
  The number of squares in a $\CAT(0)$ disc diagram is at most
  quadratic in the length $|\pbd D|$ of its boundary path.  Indeed, if
  $|\pbd D| = n$ then $D$ has less than $n$ dual curves.  Then, by
  \Propref{nonogon}, there are less than $\binom{n}{2}$ intersections
  of dual curves in $D$ and so less than $\binom{n}{2}$ squares.
\end{proof}

\begin{cor}
  Let $G$ be a quadric group.  Then $G$ has quadratic isoperimetric
  function.
\end{cor}
\begin{proof}
  Follows by quasi-isometry invariance of isoperimetric functions
  \cite{Alonso:1990}.
\end{proof}

\begin{cor}
  Let $G$ be a quadric group.  Then $G$ has decidable word problem.
\end{cor}
\begin{proof}
  Since $G$ has a quadratic isoperimetric function, it has a recursive
  isoperimetric function.  It follows that $G$ has decidable word
  problem \cite[Theorem~2.1 and Lemma~2.2]{Gersten:1991}.
\end{proof}

\subsection{\texorpdfstring{$4$-Bridged}{4-Bridged} Graphs}
\sseclabel{fbridged}

We now use the properties of disc diagrams in quadric complexes
developed in \Ssecref{catzdiscs} to characterize them by their
$1$-skeleta.

A graph $\Gamma$ is \defterm{$n$-bridged} if every isometrically
embedded cycle of $\Gamma$ has length $n$.  Graphs that are
$3$-bridged are known as \newterm{bridged graphs} in the metric graph
theory literature \cite{Anstee:1988} and $4$-bridged graphs are the same
as the \newterm{hereditary modular graphs} of metric graph theory
\cite{Bandelt:1988}.  Note that a $4$-bridged graph is simplicial and
bipartite.  Any immersed $4$-cycle or $6$-cycle in a bipartite simplicial
graph is embedded.  Any embedded $6$-cycle in a $4$-bridged graph has a
\newterm{diagonal}, i.e., an edge joining a pair of opposing vertices.

A square complex $X$ is \defterm{$4$-flag} if every embedded $4$-cycle
in $X$ bounds a unique square and that the boundary of every square is
an embedded $4$-cycle.

\begin{lem}
  \lemlabel{fbsc} Let $X$ be a $4$-flag square complex and suppose
  $X^1$ is $4$-bridged.  Then $X$ is simply connected.
\end{lem}
\begin{proof}
  It suffices to construct a disc diagram for a given cycle
  $\alpha \colon C \to X^1$.  Since the girth of $X^1$ is at least
  $4$, if $|C| < 4$ then $\alpha$ factors through a tree and so has a
  singular disc diagram with no squares.  If $|C| = 4$ then, by
  $4$-flagness, either $\alpha$ factors through a tree or $\alpha$
  bounds a square whose inclusion is a disc diagram for $\alpha$.  We
  argue now by induction on $|C| > 4$.  Since $\Gamma$ is $4$-bridged,
  our cycle $\alpha$ is not isometrically embedded.  So, for some pair
  of vertices $u,v \in C^0$, there exists a path
  $\gamma \colon P \to X^1$ from $\alpha(u)$ to $\alpha(v)$ with
  $|P| < |Q|$ and $|P| < |R|$ where $Q$ and $R$ are the two segments
  of $C$ between $u$ and $v$.  We glue $\alpha$ and $\gamma$ together
  to obtain a map $\beta \colon (C \sqcup R)/{\sim} \to X^1$ by
  identifying the initial vertex of $R$ with $u$ and the terminal
  vertex of $R$ with $v$.  Then $\beta|_{P \cup R}$ and
  $\beta|_{Q \cup R}$ are cycles of length
  $|P| + |R| < |P| + |Q| = |C|$ and $|Q| + |R| < |Q| + |P| = |C|$.
  So, by induction, we have disc diagrams $D_1$ and $D_2$ for
  $\beta|_{P \cup R}$ and $\beta|_{Q \cup R}$.  Gluing $D_1$ and $D_2$
  together along $R$ we obtain a disc diagram for $C$.
\end{proof}

\begin{prop}
  \proplabel{mgchar} Let $X$ be a square complex.  The following are
  equivalent.
  \begin{enumerate}
  \item \itmlabel{quad} $X$ is quadric.
  \item \itmlabel{brgd} $X$ is $4$-flag and the $1$-skeleton $\sk{1}{X}$
    of $X$ is $4$-bridged.
  \item \itmlabel{sxcc} $X$ is $4$-flag and simply connected and the
    $1$-skeleton $\sk{1}{X}$ of $X$ is simplicial and every $6$-cycle
    in $X^1$ has a diagonal.
  \end{enumerate}
\end{prop}
\begin{proof}
  (\itmref{quad}) $\Rightarrow$ (\itmref{brgd}) Suppose $X$ is
  quadric.  Simple connectivity of $X$ and the fact that its $2$-cells
  all have even boundary length imply that cycles in $X$ have even
  length and so $X^1$ is bipartite.  The $\CAT(0)$ property of minimal
  disc diagrams in $X$ further implies that no embedded cycle has
  length $2$ and that each embedded $4$-cycle bounds a square: no
  nonsingular $\CAT(0)$ disc diagram has boundary length $2$ and the
  only nonsingular $\CAT(0)$ disc diagram of boundary length $4$ is a
  square.  So $X$ is $4$-flag and any embedded cycle has length at
  least $4$.  It remains to show that any isometrically embedded cycle
  $\gamma$ of $X$ has length at most $4$.  Suppose $\gamma$ is
  isometrically embedded and take a locally minimal (and hence
  $\CAT(0)$) disc diagram $D$ in $X$ with boundary path $\gamma$.  Let
  $\alpha$ be a dual curve of $D$ with endoints at edges $e$ and $e'$
  of $\bd D$.  Since $\gamma$ is isometric, any subpath of $\bd D$ of
  length at most $\frac{1}{2}|\bd D|$ must be geodesic in $D$ so, by
  \Lemref{geodcurve}, no such subpath may have initial and terminal
  edges $e$ and $e'$.  Hence $e$ and $e'$ are antipodal in $\bd D$ and
  every dual curve starting at some edge in $\bd D$ ends at the
  antipodal edge in $\bd D$.  Then all such dual curves pairwise
  intersect and so, by \Propref{nonogon}, there are at most two such
  dual curves.  This implies that $|\gamma| \le 4$.

  (\itmref{brgd}) $\Rightarrow$ (\itmref{sxcc}) By \Lemref{fbsc}, our
  square complex $X$ is simply connected.  An embedded $6$-cycle of a
  bipartite graph that is not isometrically embedded must have a
  diagonal.  Hence every embedded $6$-cycle in $X^1$ has a diagonal.

  (\itmref{sxcc}) $\Rightarrow$ (\itmref{quad}) Simple connectedness
  of $X$ implies that $X^1$ is bipartite.  Conditions
  \defnitmref{quadric}{imm} and \defnitmref{quadric}{piec3} are
  satisfied by $4$-flagness and because $\sk{1}{X}$ is simplicial.
  Since, additionally, $\sk{1}{X}$ is bipartite, any immersed
  $4$-cycle or $6$-cycle in $\sk{1}{X}$ must be embedded.  Then
  Condition~\defnitmref{quadric}{rep2} follows by $4$-flagness and
  Condition~\defnitmref{quadric}{rep3} follows because every embedded
  $6$-cycle has a diagonal, which splits the $6$-cycle into two
  embedded $4$-cycles joined along the diagonal.
\end{proof}

Compare \Propref{mgchar} with the fact that a simplicial complex is
systolic if and only if it is flag (in the usual sense) and its
$1$-skeleton is $3$-bridged \cite{Chepoi:2000}.

The \defterm{$4$-flag completion} $\overline\Gamma$ of a graph $\Gamma$
is the square complex obtained by gluing a unique square to each
embedded $4$-cycle of $\Gamma$.  \Propref{mgchar} states that the $4$-flag
completion of a $4$-bridged graph is quadric and that every quadric
complex can be obtained in this way.  In other words, the map
$X \mapsto \sk{1}{X}$ is a bijection from the class of quadric
complexes to the class of $4$-bridged graphs with inverse
$\ff{\Gamma} \mapsfrom \Gamma$.

\subsection{Balls in the \texorpdfstring{$1$-Skeleton}{1-Skeleton}}
\sseclabel{ballsisom}

Balls in the $1$-skeleton of a quadric complex are not generally
convex (e.g. consider the standard square tiling of the plane).  In
fact, there exist quadric complexes wherein the convex hull of a ball
of radius $1$ is infinite, as in the following example.

\begin{ex}
  \exlabel{convhull} Let $\Gamma$ be the standard Cayley graph of
  $\Z^2$.  Let $\bar \Gamma$ be the quotient of $\Gamma$ by the $\Z$
  action given by $k \cdot (m,n) = (m+2k,n+2k)$.  Then the $4$-flag
  completion of $\bar \Gamma$ is simply connected and checking that
  every $6$-cycle has a diagonal we see, by \Propref{mgchar}, that
  $\bar \Gamma$ is $4$-bridged.  Then $\bar \Gamma$ is the
  $1$-skeleton of a quadric complex, by \Propref{mgchar}.  The path of
  length $4$ of $\Gamma$ given by $(0,0),(1,0),(1,1),(2,1),(2,2)$
  projects to a $4$-cycle in $\bar \Gamma$ whose convex hull is all of
  $\bar \Gamma$.
\end{ex}

\Exref{convhull} contrasts sharply with the case of the $1$-skeleton
of a systolic complex, where a neighborhood of any convex subgraph is
convex.  Balls in the $1$-skeleton of quadric complexes are
isometrically embedded, however, and this fact is crucial in our proof
in \Secref{invbc} of the Invariant Biclique Theorem.

A ball $B_r(v)$ of radius $r$ centered at a vertex $v$ of a graph
$\Gamma$ is the full subgraph on the set of all vertices of $\Gamma$
at distance at most $r$ to $v$.  The following lemma is an easy
corollary of a theorem of Bandelt \cite[Theorem~2(ii)]{Bandelt:1988}.
We give a disc diagrammatic proof.

\begin{lem}
  \lemlabel{ballsisom} Let $X$ be a quadric complex.  The balls of its
  $1$-skeleton $\sk{1}{X}$ are isometrically embedded subgraphs.
\end{lem}
\begin{proof}
  Let $B_r(v)$ be the ball of radius $r$ centered at some vertex $v$
  of $X$.  Let $a$ and $b$ be vertices of $B_r(v)$ and take geodesics
  $\alpha$ from $v$ to $a$ and $\beta$ from $v$ to $b$.  For each
  geodesic $\gamma$ of $\sk{1}{X}$ from $a$ to $b$ there is a minimal
  area disc diagram $D$ in $X$ with boundary path
  $\pbd D = \alpha \gamma \beta^{-1}$.  Choose $(\gamma, D)$
  minimizing the area of $D$.  Minimality implies that no interior
  vertex of $\gamma$ has curvature $\frac{\pi}{2}$.  By
  \Lemref{geodcurve}, no dual curve of $D$ both starts and ends on one
  of $\alpha$, $\beta$, or $\gamma$.

  We claim that no two dual curves starting on $\gamma$ cross.
  Suppose there exist such pairs of dual curves and choose a pair
  $\delta_1$ and $\delta_2$ minimizing the number of vertices bounded
  by $\gamma$, $\delta_1$ and $\delta_2$.  Then $\delta_1$ and
  $\delta_2$ start at adjacent edges $e_1$ and $e_2$ of $\gamma$.  Let
  $s$ be the square in which the $\delta_i$ cross.  Let $N_1$ and
  $N_2$ be the initial segments of the carriers of $\delta_1$ and
  $\delta_2$, starting at $\gamma$ and containing all squares up until
  but not including $s$.  By \Propref{carrier}, the parallel paths of
  $N_1$ and $N_2$ starting at $e_1 \cap e_2$ coincide.  So if the
  $N_i$ have nonzero length then the final vertex of this common
  parallel path is an interior vertex of $D$ incident to exactly three
  squares: the final squares of the $N_i$ and $s$.  This contradicts
  the $\CAT(0)$ condition.  So the $N_i$ have length $0$ and $s$
  contains $e_1$ and $e_2$.  But then the vertex $e_1 \cap e_2$ has
  curvature $\frac{\pi}{2}$, contradicting the minimality of
  $(\gamma, D)$.

  Therefore, $\gamma = \gamma_\alpha \gamma_\beta$ is the disjoint
  union of two subpaths $\gamma_\alpha$ and $\gamma_\beta$ where all
  dual curves of $D$ starting on $\gamma_\alpha$ end on $\alpha$ and
  all those staring on $\gamma_\beta$ end on $\beta$.  It follows, by
  \Propref{dcdistance}, that every vertex $u$ in the interior of
  $\gamma$ is closer to $v$ than $a$ (if $u \in \gamma_\alpha$) or $b$
  (if $u \in \gamma_\beta$).  That is, $\gamma \subset B_r(v)$.
\end{proof}

\subsection{Two Results of Hanlon and Martinez-Pedroza}
\sseclabel{hanped}

We end this section by stating two theorems of Hanlon and
Martinez-Pedroza that apply to quadric complexes and groups.
\begin{thm}[{\cite[Theorem 1.1]{Hanlon:2014}}]
  \thmlabel{fpsg} Let $\catname{C}$ be any category of complexes
  closed under taking full subcomplexes and covering spaces.  The
  category of groups acting properly and cocompactly on simply
  connected complexes in $\catname{C}$ is closed under taking finitely
  presented subgroups.
\end{thm}
\begin{cor}
  A finitely presented subgroup of a quadric group is quadric.
\end{cor}
\Thmref{fpsg} generalizes a theorem of Wise that any finitely
presented subgroup of the fundamental group of a ``pure
$(p,q)$-complex'' is the fundamental group of a pure $(p,q)$-complex,
which implies that finitely presented subgroups of torsion-free
systolic groups are systolic \cite{Wise:2003}.

In order to prove \Thmref{fpsg}, Hanlon and Martinez-Pedroza use the
following theorem, which we use in the proof of the Invariant Biclique
Theorem.
\begin{thm}[{\cite[Theorem 4.1]{Hanlon:2014}}]
  If $G$ is a group acting properly on a simply connected locally
  finite complex $X$ and $H$ is a finitely presented subgroup of $G$,
  then $H$ acts cocompactly on a simply connected complex $X'$ that
  maps $H$-equivariantly into $X$ through an $\mathscr{F}$-tower
  $X' \to X$.
\end{thm}
An \newterm{$\mathscr{F}$-tower} is a composition of covering maps and
inclusions of full subcomplexes and so, if $X$ is locally quadric,
then $X'$ is quadric.  We need only the following special case of the
theorem.
\begin{cor}
  \corlabel{fincore} If $G$ is a finite group acting on a simply
  connected locally finite quadric complex $X$, then $G$ acts on a
  finite quadric complex $X'$ that immerses $G$-equivariantly into
  $X$.
\end{cor}

\section{The Invariant Biclique Property}
\seclabel{invbc}

We first prove the Invariant Biclique Theorem for finite quadric
complexes and then generalize this result to locally finite quadric
complexes by applying \Corref{fincore} of Hanlon and Martinez-Pedroza.
We prove the finite version of the theorem by showing that finite
$4$-bridged graphs are bi-dismantlable---a theorem of Bandelt
\cite[Theorem~2]{Bandelt:1988}---and that each such graph has a
biclique invariant under all of its automorphisms.  Our proof of the
latter, in \Ssecref{dismantle}, follows that in Hensel et al.
\cite{Hensel:2014} showing that dismantlable graphs have invariant
cliques---original proved by Polat \cite{Polat:1993}.  In
\Ssecref{bfs} we apply the breadth-first search algorithm to show that
$4$-bridged are bi-dismantlable, a technique used by Chepoi
\cite{Chepoi:1997} to give an alternate proof of the theorem of Anstee
and Farber \cite{Anstee:1988} that bridged graphs are dismantlable.

\subsection{Dismantling Bipartite Graphs}
\sseclabel{dismantle}

A \defterm{metric sphere} $S_r(u)$ of radius $r$ centered at a vertex
$u$ in a graph $\Gamma$ is the full subgraph on the set of vertices of
$\Gamma$ at distance $r$ from $u$.  If $\Gamma$ is bipartite and
$S_r(u)$ is a metric sphere in $\Gamma$ then $S_r(u)$ has no edges.
The neighbours of any vertex $v \in S_r(u)$ are in $S_{r-1}(u)$ and
$S_{r+1}(u)$.

\begin{defn}
  \defnlabel{dismantle} Let $\Gamma$ be a finite bipartite simplicial
  graph.  If $u$ and $v$ are distinct vertices of $\Gamma$ then $u$ is
  \defterm{bi-dominated} by $v$ if every neighbour of $u$ is a
  neighbour of $v$, i.e., if there is a containment
  $S_1(u) \subset S_1(v)$ of metric spheres.

  $\Gamma$ is \defterm{bi-dismantlable} if there exists a sequence
  \[\Gamma = \Gamma_1, \Gamma_2, \ldots, \Gamma_n\] of graphs ending on
  a biclique such that, for each $i < n$,
  $\Gamma_{i+1} = \Gamma_i \setminus v_i$ for some $v_i$ bi-dominated
  in $\Gamma_i$.  In other words, $\Gamma$ is bi-dismantlable if we can
  obtain a biclique from $\Gamma$ by successively removing bi-dominated
  vertices.
\end{defn}

These definitions are modified from the standard ones, which require a
dominated vertex to be a neighbour of its dominator and require a
dismantlement to end on a clique.  The modification is necessary to
work with bipartite graphs which, by the original definitions, are
dismantlable only in the case of trees.

Note that if $u$ bi-dominates $v$ in $\Gamma$ then $u$ bi-dominates $v$ in
every full subgraph of $\Gamma$ containing $u$ and $v$.  The
bi-domination relation is transitive.  If $u$ bi-dominates $v$ which
bi-dominates $w$ then $u$ bi-dominates $w$.

At each step of a bi-dismantlement of a graph, there may be several
bi-dominated vertices that could potentially be removed.  The following
lemma and proof, which are adapted from Lemma~2.5 of Hensel et
al. \cite{Hensel:2014}, show that these choices can be made
arbitrarily.

\begin{lem}
  \lemlabel{greedism} If $\Gamma$ is a finite bi-dismantlable
  simplicial graph and a vertex $v$ is bi-dominated in $\Gamma$ by a
  vertex $u$, then $\Gamma \setminus v$ is bi-dismantlable.
\end{lem}
\begin{proof}
  Let $\Gamma = \Gamma_1, \Gamma_2, \ldots, \Gamma_n$ and $v_1, v_2,
  \ldots, v_n$ be as in \Defnref{dismantle}.  If $n = 1$ then $\Gamma$
  is a biclique and then so is $\Gamma \setminus v$ so that we are
  done.  Also, if $v_1 = v$ then $\Gamma \setminus v = \Gamma_2$,
  which we know is bi-dismantlable.  So we may assume that $n > 1$ and
  that $v \neq v_1$.
  
  If $v_1$ is bi-dominated in $\Gamma \setminus v$ then we can remove it
  to obtain $\Gamma_2 \setminus v$.  And if $v$ is bi-dominated in
  $\Gamma_2$ then $\Gamma_2 \setminus v$ is bi-dismantlable by
  induction on $n$.  So if both of these conditions hold, then $\Gamma
  \setminus v$ is bi-dismantlable and we are done.  We assume that
  either: (I) $v_1$ is not bi-dominated in $\Gamma \setminus v$ or (II)
  $v$ is not bi-dominated in $\Gamma_2$.
  
  \paragraph{Case I: $v_1$ is not bi-dominated in $\Gamma \setminus v$}
  Then $v$ must have been the only bi-dominator of $v_1$ in $\Gamma$.  It
  follows that $v_1 = u$ (otherwise $u$ bi-dominates $v_1$ in $\Gamma
  \setminus v$).

  \paragraph{Case II: $v$ is not bi-dominated in $\Gamma_2$} Then $v_1 =
  u$ and the bi-dominator of $v_1$ in $\Gamma$ must be $v$ (or else it
  bi-dominates $v$ in $\Gamma_2$).
  
  In either case we find that $v$ bi-dominates $u$ in $\Gamma$ and that
  $v_1 = u$.  But $u$ bi-dominates $v$ in $\Gamma$ so $u$ and $v$ have
  the same set of neighbours.  Hence $\Gamma \setminus v \isomor
  \Gamma \setminus u = \Gamma \setminus v_1 = \Gamma_2$, which is
  bi-dismantlable.
\end{proof}

The following theorem is the main result of this section, the
invariant biclique theorem for finite bi-dismantlable graphs.  Its
proof is adapted from Theorem~2.4 of Hensel et al. \cite{Hensel:2014}.

\begin{thm}
  \thmlabel{disminvbc} Let $\Gamma$ be a finite bi-dismantlable
  simplicial graph that is not a single vertex.  Then the automorphism
  group $G$ of $\Gamma$ stabilizes a biclique of $\Gamma$.
\end{thm}
\begin{proof}
  The proof is by induction on the number $\#\sk{0}{\Gamma}$ of
  vertices of $\Gamma$.  If $\#\sk{0}{\Gamma} = 2$ then $\Gamma$ is a
  biclique.  This proves the base case.  For $\#\sk{0}{\Gamma} > 2$
  let $D$ be the set of bi-dominated vertices of $\Gamma$ and consider
  the following two cases.
  
  \paragraph{Case I: No two vertices of $D$ have the same set of
    neighbours} It follows that every bi-dominated vertex $v \in D$ has a
  bi-dominator in $\Gamma \setminus D$.  Otherwise we could construct a
  sequence $v = v_1, v_2, v_3, \ldots$ in $D$ starting on $v$ such
  that $v_i$ is bi-dominated by $v_{i+1}$, for each $i$.  But $D$ is
  finite so this sequence must have repeated vertices, contradicting
  the assumption of this case.  Now $D$ is invariant under $G$ and
  hence so is the full subgraph $\Gamma' = \Gamma \setminus D$.  But
  $\Gamma'$ is bi-dismantlable, by \Lemref{greedism}, and so stabilizes
  a biclique, by induction.

  \paragraph{Case II: There is a pair of vertices of $D$ that have the
    same set of neighbours} Define a relation $\sim$ on the vertex set
  of $\Gamma$ by $u \sim v$ if $u$ and $v$ have the same set of
  neighbours ($S_1(u) = S_1(v)$).  This is an equivalence relation so
  we can take the quotient graph $\Gamma' = \Gamma/{\sim}$.  The
  quotient $\Gamma'$ can also be obtained by successively removing
  bi-dominated vertices starting with $\Gamma$, since a vertex of
  $\Gamma$ is bi-dominated by every other member of its equivalence
  class.  So, by \Lemref{greedism}, $\Gamma'$ is bi-dismantlable.  The
  action of $G$ descends to an action on $\Gamma'$ and the quotient
  map $\Gamma \to \Gamma'$ is $G$-equivariant.  But $\Gamma'$ has
  fewer vertices than $\Gamma$ and so, by induction, has a stable
  biclique $B$.  The preimage of an edge in $\Gamma'$ is a biclique in
  $\Gamma$ and so the preimage of $B$ is an invariant biclique of
  $\Gamma$.
\end{proof}

\subsection{Breadth-First Search and
  \texorpdfstring{$4$-Bridged}{4-Bridged} Graphs}
\sseclabel{bfs}

In order to show that finite $4$-bridged graphs are bi-dismantlable we
present a well-known algorithmic tool, the breadth-first search.  This
algorithm gives us a spanning tree and a numbering of the vertices of
a $4$-bridged graph which we can use to dismantle it, as in Chepoi
\cite{Chepoi:1997} for the case of bridged graphs.  Bandelt first
proved the bi-dismantlability of $4$-bridged graphs
\cite[Theorem~2]{Bandelt:1988}.

\begin{defn}
  \defnlabel{bfs} \defterm{Breadth-first search} (BFS) is an algorithm
  taking as input a locally finite simplicial graph $\Gamma$ on $n$
  vertices and a starting vertex $u \in \Gamma$.  The BFS algorithm
  visits all of the vertices of $\Gamma$, starting with $u$, numbering
  the $i$th vertex it visits with the number $i$.  The algorithm
  proceeds as follows.
  \begin{enumerate}
  \item Assign the number $1$ to $u$.
  \item As long as there remains a vertex of $\Gamma$ that is not
    assigned a number, repeat the following two steps.
    \begin{enumerate}
    \item \itmlabel{leastnumv} Let $v$ be the least numbered vertex of
      $\Gamma$ that has an unnumbered neighbour.
    \item \itmlabel{assnum} Take an arbitrary unnumbered neighbour
      $w$ of $v$ and assign to it the next available number
      (i.e. the least positive integer not yet assigned to a
      vertex).
    \end{enumerate}
  \end{enumerate}
  The strict total order induced on the vertices of $\Gamma$ by these
  numbers is called the \defterm{BFS order}.  We denote a BFS order
  with starting vertex $u$ by $\bfslt_u$ or just $\bfslt$.
\end{defn}

A BFS order with starting vertex $u$ is not necessarily unique due to
the arbitrary choices made during an execution of the BFS algorithm.
However, the order is always compatible with the distance of vertices
from $u$.  Precisely, if the vertex $v$ is closer to $u$ than the
vertex $w$ then $v$ precedes $w$ in any BFS order with starting vertex
$u$, i.e., \[ |v,u| < |w,u| \implies v \bfslt_u w. \]

Fix a BFS order on $\Gamma$ with starting vertex $u$.  A neighbour $w$
of a vertex $v$ is called a \defterm{pseudoparent} of $v$ if it
precedes $v$ in the BFS order.  So if $v$ and $w$ are adjacent
vertices of $\Gamma$ then one of $v$ or $w$ is a pseudoparent of the
other.  The minimal pseudoparent of $w$ in the BFS order is the
\defterm{parent} of $w$, denoted $\p w$.  Every vertex of $\Gamma$
other than the starting vertex $u$ has a parent.  If $v \bfslt w$ for
some vertices $v$ and $w$ of $\Gamma$ having parents $\p v$ and $\p w$
then $\p v \bfsle \p w$.  The subgraph $\Gamma'$ of $\Gamma$ defined
by the parent relation is a spanning tree of $\Gamma$.  If $\Gamma$ is
bipartite, then the pseudoparents of a vertex $v$ are all strictly
closer to $u$ than $v$ is.

\begin{figure}
  \centering
  \begin{subfigure}[b]{0.33\textwidth}
    \begin{tikzpicture}
      \coordinate (v) at (0,0);
      \coordinate (Pv) at (-1,1);
      \coordinate (w) at (1,1);
      \coordinate (x) at (0,2);
      \coordinate (P2v) at (-1,2);
      \coordinate (Pw) at (1,2);

      \begin{scope}[thick]
        \draw (v) -- (w);
        \draw (Pv) -- (x);
        \draw (w) -- (x);
        \begin{scope}[decoration={
            markings,
            mark=at position 0.5 with {\arrow{>}}}] 
          \draw[postaction={decorate}] (v) -- (Pv);
          \draw[postaction={decorate}] (w) -- (Pw);
          \draw[postaction={decorate}] (Pv) -- (P2v);
        \end{scope}
      \end{scope}

      \node[vertex,label={below:$v$}] at (v) {};
      \node[vertex,label={left:$Pv$}] at (Pv) {};
      \node[vertex,label={right:$w$}] at (w) {};
      \node[vertex,label={above:$x$}] at (x) {};
      \node[vertex,label={above:$P^2v$}] at (P2v) {};
      \node[vertex,label={above:$Pw$}] at (Pw) {};
    \end{tikzpicture}
    \caption{}
    \figlabel{fellowtrav}
  \end{subfigure}%
  \begin{subfigure}[b]{0.33\textwidth}
    \begin{tikzpicture}
      \coordinate (v) at (0,0);
      \coordinate (Pv) at (-1,1);
      \coordinate (w) at (1,1);
      \coordinate (x) at (0,2);
      \coordinate (P2v) at (-1,2);
      \coordinate (Pw) at (1,2);
      \coordinate (Px) at (0,3);

      \begin{scope}[thick]
        \draw (v) -- (w);
        \draw (Pv) -- (x);
        \draw (w) -- (x);
        \draw (P2v) -- (Px);
        \draw (Pw) -- (Px);
        \begin{scope}[decoration={
            markings,
            mark=at position 0.5 with {\arrow{>}}}] 
          \draw[postaction={decorate}] (v) -- (Pv);
          \draw[postaction={decorate}] (w) -- (Pw);
          \draw[postaction={decorate}] (Pv) -- (P2v);
          \draw[postaction={decorate}] (x) -- (Px);
        \end{scope}
      \end{scope}

      \node[vertex,label={below:$v$}] at (v) {};
      \node[vertex,label={left:$Pv$}] at (Pv) {};
      \node[vertex,label={right:$w$}] at (w) {};
      \node[vertex,label={right:$x$}] at (x) {};
      \node[vertex,label={left:$P^2v$}] at (P2v) {};
      \node[vertex,label={right:$Pw$}] at (Pw) {};
      \node[vertex,label={above:$Px$}] at (Px) {};
    \end{tikzpicture}
    \caption{}
    \figlabel{fellowtrav2}
  \end{subfigure}%
  \begin{subfigure}[b]{0.33\textwidth}
    \begin{tikzpicture}
      \coordinate (v) at (0,0);
      \coordinate (Pv) at (-1,1);
      \coordinate (y) at (0,1);
      \coordinate (w) at (1,1);
      \coordinate (Py) at (-1,2);
      \coordinate (Pw) at (1,2);
      \coordinate (x) at (0,3);

      \begin{scope}[thick]
        \draw (v) -- (y);
        \draw (v) -- (w);
        \draw (Pv) -- (Py);
        \draw (Pv) -- (Pw);
        \draw (Py) -- (x);
        \draw (Pw) -- (x);
        \draw[thickeraser] (y) -- (Py);
        \begin{scope}[decoration={
            markings,
            mark=at position 0.5 with {\arrow{>}}}] 
          \draw[postaction={decorate}] (v) -- (Pv);
          \draw[postaction={decorate}] (y) -- (Py);
          \draw[postaction={decorate}] (w) -- (Pw);
        \end{scope}
      \end{scope}

      \node[vertex,label={below:$v$}] at (v) {};
      \node[vertex,label={left:$Pv$}] at (Pv) {};
      \node[vertex,label={right:$y$}] at (y) {};
      \node[vertex,label={right:$w$}] at (w) {};
      \node[vertex,label={left:$Py$}] at (Py) {};
      \node[vertex,label={right:$Pw$}] at (Pw) {};
      \node[vertex,label={above:$x$}] at (x) {};
    \end{tikzpicture}
    \caption{}
    \figlabel{bfsdom}
  \end{subfigure}
  \caption{Figures for \Lemref{fellowtrav} and \Lemref{bfsdom}.}
\end{figure}

\begin{lem}
  \lemlabel{fellowtrav} Let $\Gamma$ be a locally finite $4$-bridged
  graph and run the BFS algorithm on $\Gamma$ starting at an arbitrary
  vertex $u$.  Let $v$ and $w$ be vertices of $\Gamma$ with parents
  $\p v$ and $\p w$.  If $w$ is a pseudoparent of $v$, then $\p w$ is
  a pseudoparent of $\p v$.
\end{lem}
\begin{proof}
  The proof is by induction on $|v,u|$, which is at least $2$ since
  $w$ is assumed to have a parent.  If $|v,u| = 2$ then
  $\p w = u = \p^2 v$ and so the lemma holds.  Suppose the lemma holds
  for $2 \le |v,u| < r$.  We will prove it for $|v,u| = r$.  If
  $w = \p v$ then $\p w = \p^2 v$ is the parent, and so a
  pseudoparent, of $\p v$.  So we may assume that $w \neq \p v$.  By
  the definition of parent, this implies that $\p v \bfslt w$.  Since
  $\Gamma$ is bipartite, it also implies that $|\p v, w| = 2$ in
  $\Gamma$.  But the ball $B_{r-1}(u)$ contains $\p v$ and $w$ so, by
  \Lemref{ballsisom}, $\p v$ and $w$ have a common pseudoparent $x$.
  
  Now, if $\p w = x$ or $\p w = \p^2 v$ then we are done, so assume
  that $\p w \neq x$ and $\p w \neq \p^2 v$.  Then, since
  $\p v \bfslt w$, we have $\p^2 v \bfslt \p w \bfslt x$ and so
  $\p^2 v \neq x$.  Hence the graph in \Figref{fellowtrav} is embedded
  in $\Gamma$.  By induction, the parent $\p x$ of $x$ is a
  pseudoparent of both $\p w$ and $\p^2 v$ as in the graph in
  \Figref{fellowtrav2} which is also embedded in $\Gamma$.  Then, as
  $\Gamma$ is $4$-bridged, the $6$-cycle $v$, $\p v$, $\p^2 v$, $\p x$,
  $\p w$, $w$, $v$ must have a diagonal.  But $\p x$ cannot be
  adjacent to $v$, as $|\p x, u| = |v,u| - 3$, and $\p^2 v$ cannot be
  adjacent to $w$, since $\p^2 v \bfslt \p w$ and $\p w$ is the parent
  of $w$.  Hence $\p w$ and $\p v$ must be adjacent so that $\p w$ is
  a pseudoparent of $\p v$.
\end{proof}

\begin{lem}
  \lemlabel{bfsdom} Let $\Gamma$ be a locally finite $4$-bridged graph
  and run the BFS algorithm on $\Gamma$ starting at some vertex $u$.
  Let $x$ be any vertex of $\Gamma$ at distance $r = |x,u| \ge 2$ from
  $u$.  Then $x$ is bi-dominated in $B_r(u)$ by the parent of its
  $\bfslt$-maximal pseudoparent.
\end{lem}
\begin{proof}
  Let $w$ be the maximal pseudoparent of $v$ in the BFS order.  We
  want to show that every neighbour of $v$ in $B_r(u)$ is also a
  neighbour of $\p w$.  But the neighbours of $v$ in $B_r(u)$ are
  precisely its pseudoparents.  So it suffices to show that $\p w$ is
  adjacent to every pseudoparent of $v$.  This holds for the
  pseudoparents $\p v$ and $w$, by \Lemref{fellowtrav}.
  
  Let $y$ be any other pseudoparent of $v$.  We will show that $\p w$
  is adjacent to $y$.  By \Lemref{fellowtrav}, $\p y$ is a
  pseudoparent of $\p v$.  If $\p y$ and $\p w$ coincide then we are
  done, so we may assume otherwise.  Then $|\p y, \p w| = 2$ in
  $\Gamma$ and so, by \Lemref{ballsisom}, $\p y$ and $\p w$ have a
  common pseudoparent $x$ as shown in the graph in \Figref{bfsdom}
  which is embedded in $\Gamma$.  Then, as $\Gamma$ is $4$-bridged, the
  $6$-cycle $v, y, \p y, x, \p w, w, v$ must have a diagonal.  But $x$
  and $v$ cannot be adjacent since $x$ is too close to $u$ relative to
  $v$.  Nor can $\p y$ and $w$ be adjacent because the facts
  $y \bfslt w$ and $\p y \neq \p w$ imply $\p y \bfslt \p w$ and yet
  $\p w$ is the parent of $w$.  Hence $\p w$ and $y$ are adjacent.
\end{proof}

\begin{cor}
  \corlabel{fbdism} Let $\Gamma$ be a finite $4$-bridged graph.  If
  $\Gamma$ has at least two vertices, then $\Gamma$ is
  bi-dismantlable.
\end{cor}
\begin{proof}
  Take any vertex $u$ of $\Gamma$.  By \Lemref{bfsdom}, for $r \ge 2$,
  every vertex $v$ of the metric sphere $S_r(u)$ is bi-dominated in
  $B_r(u)$ by some vertex $w$ in $S_{r-2}(u)$.  Hence, we may
  successively remove vertices at maximal distance from $u$, until the
  biclique $B_1(u)$ is all that remains.
\end{proof}

\begin{cor}
  \corlabel{finfb} A group $G$ acting on a finite quadric complex $X$
  not equal to a single vertex, stabilizes a biclique of $X$.
\end{cor}
\begin{proof}
  The $1$-skeleton $\sk{1}{X}$ of $X$ is $4$-bridged, by \Propref{mgchar},
  and every automorphism of $X$ restricts to an automorphism of
  $\sk{1}{X}$ which is bi-dismantlable, by \Corref{fbdism}.  So by
  \Thmref{disminvbc}, $X$ has an invariant biclique.
\end{proof}

\begin{thm}[Invariant Biclique Theorem]
  \thmlabel{ibt} Let $G$ be a finite group acting on a locally finite
  quadric complex $X$, which is not equal to a single vertex.  Then
  $G$ stabilizes a biclique of $X$.
\end{thm}
\begin{proof}
  By \Corref{fincore} of Hanlon and Martinez-Pedroza, we have a finite
  quadric complex $X'$ mapping $G$-equivariantly into $X$. By
  \Corref{finfb}, $G$ stabilizes a biclique of $X'$.  The image of
  this biclique is an invariant biclique of $X$.
\end{proof}

As a corollary we obtain the following.

\begin{cor}
  Let $G$ be a quadric group.  Then $G$ has finitely many conjugacy
  classes of finite groups.
\end{cor}
\begin{proof}
  Let $X$ be a quadric complex on which $G$ acts properly and
  cocompactly.  By \Thmref{ibt}, each finite subgroup of $G$ is a
  subgroup of the stabilizer of some biclique of $X$.  Since the
  action is proper, the stabilizer of each biclique is finite and so
  has finitely many finite subgroups.  So it suffices to show that $X$
  has finitely many orbits of bicliques.  But every biclique is
  contained in a ball of radius $2$ of $\sk{1}{X}$.  Balls of radius
  $2$ are finite by local finiteness and, by cocompactness, there are
  finitely many orbits of such balls.  Hence, there are finitely many
  orbits of bicliques.
\end{proof}

\section{\texorpdfstring{$\cftf$}{C(4)-T(4)} Groups}
\seclabel{cftfgp}

Small cancellation theory traces its origins to the work of Dehn on
hyperbolic surface groups \cite{Dehn:1987} and has played a
significant role in the study of infinite groups since then.  The
standard textbook on the subject is Lyndon and Schupp
\cite{Lyndon:2001}.  We recall the definition of the $\scC(p)$ and
$\scT(q)$ small cancellation properties in \Ssecref{cptq}.
\Ssecref{cftfcomp} develops disc diagrammatic properties and the
Strong Helly Property of $\cftf$ complexes.  These properties are
crucial in the proof that $\cftf$ groups are quadric, given in
\Ssecref{quadrization}.  The proof relies on a construction
associating a square complex $X_Y$ to a simply connected $2$-complex
$Y$.  We prove that this square complex, on which the automorphism
group of $Y$ acts, is always simply connected and that it is quadric
when $Y$ is $\cftf$.

\subsection{The \texorpdfstring{$\scC(p)$}{C(p)} and \texorpdfstring{$\scT(q)$}{T(q)} Properties}
\sseclabel{cptq}

Definitions of the $\scC(p)$ and $\scT(q)$ properties and disc
diagrammatic consequences can be found elsewhere \cite{Lyndon:2001,
  McCammond:2002}.  We include them here for completeness and because
we require a Greendlinger's Lemma for $\cftf$ Disc Diagrams
(\Corref{greendlingercftf}) that gives four sites of positive
curvature on the boundary.

Before defining the $\scC(p)$ and $\scT(q)$ properties we give some
supporting definitions.

\begin{defn}
  \defnlabel{arc} Let $D$ be a disc diagram.  An \defterm{arc}
  $\alpha$ of $D$ is a minimal subgraph of its $1$-skeleton $\sk{1}{D}$
  satisfying the following conditions.
  \begin{enumerate}
  \item There is at least one $1$-cell in $\alpha$.
  \item If $\alpha$ contains a $0$-cell of valence $2$ of $\sk{1}{D}$,
    then it contains both of its incident $1$-cells.
  \end{enumerate}
  A $0$-cell $v$ of $D$ is a \defterm{node} of $D$ if it has valence
  other than $2$.  If a node $v$ is incident to an arc $\alpha$ then
  $v$ is a \defterm{node} of $\alpha$.
\end{defn}

\begin{figure}
  \centering
  \begin{tikzpicture}
    \pgfmathsetmacro\numsegments{7}
    \pgfmathsetmacro\radius{1}
    \pgfmathsetmacro\seglen{3/4}
    
    \coordinate (e0) at (0:\radius);
    \coordinate (e1) at (0:\radius+\seglen);

    \begin{scope}[thick]
      \draw (e0) arc (0:360:\radius);
      \draw (e0) -- node[above] {$e$} (e1);
    \end{scope}

    \node at (0,0) {$F$};
    
    \node[vertex] at (e0) {};
    \node[vertex] at (e1) {};
    
    \pgfmathsetmacro\nsmo{\numsegments - 1}
    \foreach \i in {1,...,\nsmo}
    {
      \pgfmathsetmacro\angle{\i * 360 / \numsegments}
      \node[vertex] at (\angle:\radius) {};
    }
  \end{tikzpicture}
  \caption{A $2$-cell with a $1$-cell attached.}
  \figlabel{spurarcs}
\end{figure}

The arcs of a disc diagram $D$ are, possibly closed, paths.  If $D$ is
a single $2$-cell then its only arc has no nodes.  If $D$ is a single
$2$-cell $F$ with a $1$-cell $e$ attached, as in \Figref{spurarcs}
then it has two arcs: a closed arc containing the boundary of $F$ and
an arc containing only $e$.  The former arc has a single node, the
valence $3$ $0$-cell of $e$.  The latter has both endpoints of $e$ as
nodes.

If an arc $\alpha$ of a disc diagram $D$ contains a $1$-cell on the
boundary $\bd D$ of $D$ then $\alpha$ lies entirely on $\bd D$.  If
$\alpha$ contains a $1$-cell incident on its two sides to the $2$-cells
$F_1$ and $F_2$ (with $F_1$ and $F_2$ possibly the same $2$-cell) then
this holds for every $1$-cell of $\alpha$.

An arc of a disc diagram $D$ is a \defterm{boundary arc} if it lies on
the boundary $\bd D$ of $D$.  A boundary arc with a valence $1$ node
is a \defterm{spur}.  The valence $1$ nodes of a spur are called its
\defterm{tips}.  A non-boundary arc is an \defterm{internal arc}.  If
an arc contains a $1$-cell that is contained in the boundary of a
$2$-cell $F$, then it is an \defterm{arc of the $2$-cell} $F$.

Let $D$ be a disc diagram in a $2$-complex $Y$ and let $F_1$ and $F_2$
be two intersecting $2$-cells of $D$.  An arc $\alpha$ of $F_1$ and
$F_2$ is \defterm{foldable} if, when the boundary paths of $F_1$ and
$F_2$ are based and oriented so as to coincide on the length of
$\alpha$, they map to the same closed path of $Y$.  A disc diagram $D$
in a $2$-complex is \defterm{reduced} if it has no foldable arcs.  Given
a disc diagram $D$ in a $2$-complex $Y$ we can obtain a reduced disc
diagram $D'$ in $Y$ with the same boundary path by performing a finite
number of reductions starting from $D$.  Each \newterm{reduction}
removes a pair of $2$-cells $F_1$ and $F_2$ that have a foldable arc.

\newcommand{\drawpetal}[8]{
  \begingroup

  \pgfmathsetmacro\numsegments{#4}
  \pgfmathsetmacro\radius{#5}
  \pgfmathsetmacro\oradius{#6}
  \pgfmathsetmacro\crado{#7}
  \pgfmathsetmacro\pca{#8}
  \pgfmathsetmacro\ps{#1}
  \pgfmathsetmacro\pe{#2}
  \pgfmathsetmacro\psegs{#3}
  \pgfmathsetmacro\pm{(\ps+\pe)/2}
  \pgfmathsetmacro\pco{(\pm+\ps)/2}
  \pgfmathsetmacro\pct{(\pm+\pe)/2}
  \pgfmathsetmacro\psa{\ps * 360 / \numsegments}
  \pgfmathsetmacro\pea{\pe * 360 / \numsegments}
  \pgfmathsetmacro\pma{\pm * 360 / \numsegments}
  \pgfmathsetmacro\pcao{\pma - \pca}
  \pgfmathsetmacro\pcat{\pma + \pca}

  \pgfmathsetmacro\cradt{\oradius/cos(\pca)}

  \draw (\psa:\radius)
    .. controls (\psa:\crado) and (\pcao:\cradt)
    .. (\pma:\oradius)
    .. controls (\pcat:\cradt) and (\pea:\crado)
    .. (\pea:\radius);
  
  \pgfmathsetmacro\psmo{\psegs-1}
  \foreach \p in {1,...,\psmo} {
    \draw[decorate,decoration={markings,mark=at
           position \p/\psegs with {\node[vertex] {};}}]
      (\psa:\radius)
      .. controls (\psa:\crado) and (\pcao:\cradt)
      .. (\pma:\oradius)
      .. controls (\pcat:\cradt) and (\pea:\crado)
      .. (\pea:\radius);
  }
  \endgroup
}
  
\begin{figure}
  \centering
  \begin{subfigure}[b]{0.35\textwidth}
    \centering
    \begin{tikzpicture}
      \pgfmathsetmacro\numsegments{12}
      \pgfmathsetmacro\radius{1}
      \pgfmathsetmacro\oradius{2 * \radius}
      \pgfmathsetmacro\crado{(\radius+\oradius)/2}
      \pgfmathsetmacro\pca{25}
      
      \begin{scope}[thick]
        \draw (0:\radius) arc (0:360:\radius);
        \drawpetal{0}{2}{1}{\numsegments}{\radius}{\oradius}{\crado}{\pca}
        \drawpetal{2}{4}{2}{\numsegments}{\radius}{\oradius}{\crado}{\pca}
        \drawpetal{4}{7}{3}{\numsegments}{\radius}{\oradius}{\crado}{\pca}
        \drawpetal{7}{9}{1}{\numsegments}{\radius}{\oradius}{\crado}{\pca}
        \drawpetal{9}{12}{4}{\numsegments}{\radius}{\oradius}{\crado}{\pca}
      \end{scope}

      \foreach \i in {1,...,\numsegments}
      {
        \pgfmathsetmacro\angle{\i * 360 / \numsegments}
        \node[vertex] at (\angle:\radius) {};
      }
    \end{tikzpicture}
    \caption{A daisy with five petals.}
    \figlabel{daisy}
  \end{subfigure}
  \hspace{1cm}%
  \begin{subfigure}[b]{0.35\textwidth}
    \centering
    \begin{tikzpicture}
      \pgfmathsetmacro\numsegments{6}
      \pgfmathsetmacro\radius{1/2}
      \pgfmathsetmacro\oradius{4 * \radius}
      \pgfmathsetmacro\crado{(\radius+\oradius)/2}
      \pgfmathsetmacro\pca{10}
      
      \begin{scope}[thick]
        \drawpetal{0}{1}{1}{\numsegments}{\radius}{\oradius}{\crado}{\pca}
        \drawpetal{1}{2}{2}{\numsegments}{\radius}{\oradius}{\crado}{\pca}
        \drawpetal{2}{3}{3}{\numsegments}{\radius}{\oradius}{\crado}{\pca}
        \drawpetal{3}{4}{1}{\numsegments}{\radius}{\oradius}{\crado}{\pca}
        \drawpetal{4}{5}{4}{\numsegments}{\radius}{\oradius}{\crado}{\pca}
        \drawpetal{5}{6}{2}{\numsegments}{\radius}{\oradius}{\crado}{\pca}
        
        \foreach \i in {1,...,\numsegments}
        {
          \pgfmathsetmacro\angle{\i * 360 / \numsegments}
          \draw (0,0) -- (\angle:\radius);
          \node[vertex] at (\angle:\radius) {};
        }
      \end{scope}
    \end{tikzpicture}
    \caption{A jasmine with six petals.}
    \figlabel{jasmine}
  \end{subfigure}
  \caption{Examples of forms of disk diagrams used in the definition
    of small cancellation conditions.}
\end{figure}

\begin{defn*}
  A \defterm{daisy} $D$ with $p$ petals is a disc diagram satisfying
  the following conditions.
  \begin{enumerate}
  \item It has a \defterm{central $2$-cell} $F$ whose boundary is
    covered by $p$ arcs.
  \item Each arc of $F$ is given by its intersection with a $2$-cell of
    $D$ called a \defterm{petal}, whose remaining boundary is a
    boundary arc.
  \item The boundary $\bd D$ of $D$ is the concatenation of the
    boundary arcs of its petals.
  \end{enumerate}
  See \Figref{daisy} for an example of a daisy.
\end{defn*}

\begin{defn*}
  For $q \ge 3$, a \defterm{jasmine} $D$ with $q$ petals is a disc
  diagram satisfying the following conditions.
  \begin{enumerate}
  \item There is a single internal $0$-cell $v$ in $D$.
  \item There are $q$ $2$-cells in $D$ called \defterm{petals}, all of
    which are incident to $v$.
  \item The internal arcs of $D$ are all embedded $1$-cells, each of
    which is incident to $v$.
  \item The boundary $\bd D$ of $D$ is the concatenation of the
    boundary arcs of its petals.
  \end{enumerate}
  See \Figref{jasmine} for an example of a jasmine.
\end{defn*}

A $2$-complex $Y$ satisfies the \defterm{$\scC(p)$ property} if its
$2$-cells are immersed and any daisy in $Y$ with fewer than $p$ petals
is foldable along an arc of its central $2$-cell.  A $2$-complex $Y$
satisfies the \defterm{$\scT(q)$ property} if its $2$-cells are immersed
and any jasmine in $Y$ with fewer than $q$ petals is foldable along
one of its internal arcs.  A group presentation $\gpres{\mcS}{\mcR}$
satisfies $\scC(p)$ or $\scT(q)$ if its associated complex does.  Note
that if $D$ is a reduced disc diagram in a $\scC(p)$ (resp. $\scT(q)$)
complex, then $D$ is $\scC(p)$ (resp. $\scT(q)$).  If $D$ is a disc
diagram in a $2$-complex of minimal area or one with the least number of
cells of any dimension then it is reduced.  If a $2$-complex $Y$ is
$\scC(p)$ or $\scT(q)$ then so is its universal cover $\ucov{Y}$.  If
a group $G$ has a finite $\scC(p)$ (resp. $\scT(q)$) presentation
$Y = \gpres{\mcS}{\mcR}$ then it acts freely and cocompactly on a
simply connected $\scC(p)$ (resp. $\scT(q)$) complex, namely the
universal cover $\ucov{Y}$ of the presentation.

\subsubsection{\texorpdfstring{$\cftf$}{C(4)-T(4)} Complexes}
\ssseclabel{cftfcomp}

A $\cftf$ complex is one satisfying both the $\scC(4)$ and $\scT(4)$
properties.  These include, for example, $\CAT(0)$ square complexes.
We develop well-known disc diagrammatic tools in this section for the
study of $\cftf$ complexes.  We use these tools to present standard
proofs that $2$-cells of $\cftf$ complexes are embedded and of the
Strong Helly Theorem, which is required in the proof that $\cftf$
groups are quadric.

Let $D$ be a $\cftf$ disc diagram.  The \defterm{curvature of a
  $2$-cell} $F$ of $D$ is
\[ \kappa(F) = 2\pi - \frac{\nu(F)}{2}\pi, \]
where $\nu(F)$ is the number of nodes on the boundary of $F$.  Let
$\delta(v)$ denote the valence of a node $v$ and $\rho(v)$ the number
of $2$-cells incident to $v$.  The \defterm{curvature of a node} $v$ of
$D$ is
\[ \kappa(v) = 2\pi - \delta(v)\pi + \frac{\rho(v)}{2}\pi. \]

\begin{prop}[Gauss-Bonnet Theorem for $\cftf$ Disc Diagrams]
  \proplabel{gbcftf} Let $D$ be a disc diagram.  The sum of the
  curvatures of nodes and $2$-cells of $D$ is $2\pi$, i.e.,
  \[ \sum_{\text{$F$ $2$-cell}} \kappa(F) + \sum_{\text{$v$ node}}
  \kappa(v) = 2\pi.\]
\end{prop}
\begin{proof}
  Let $D^{(N)}$, $D^{(A)}$ and $D^{(2)}$ be the set of nodes, arcs and
  $2$-cells of $D$.  Then the proposition follows from the below
  computation.
  \begin{align*}
    1 = \chi(D) &= \sum_{v \in D^{(N)}} 1 - \sum_{\alpha \in D^{(A)}} 1
                  + \sum_{F \in D^{(2)}} 1 \\
                &= \sum_{v \in D^{(N)}} \biggl(1 - \frac{\delta(v)}{2}\biggr)
                  + \sum_{F \in D^{(2)}} 1 \\
                &= \sum_{v \in D^{(N)}} \biggl(1 - \frac{\delta(v)}{2}
                  + \frac{\rho(v)}{4}\biggr) +
                  \sum_{F \in D^{(2)}} \biggl(1 - \frac{\nu(F)}{4}\biggr) \\
                &= \frac{1}{2\pi}\sum_{v \in D^{(N)}} \kappa(v) +
                  \frac{1}{2\pi}\sum_{F \in D^{(2)}} \kappa(F)
  \end{align*}
\end{proof}

Let $D$ be a disc diagram.  A \defterm{boundary $2$-cell} of $D$ is a
$2$-cell of $D$ with a boundary arc.  A \defterm{site of positive
  curvature} on the boundary $\bd D$ of a disc diagram $D$ is the tip
of a spur or a boundary $2$-cell with fewer than four nodes on its
boundary.

\begin{cor}[Greendlinger's Lemma for $\cftf$ Disc Diagrams]
  \corlabel{greendlingercftf} Let $D$ be a $\cftf$ disc diagram and
  assume that $D$ is not a single $0$-cell or $2$-cell.  Then there are at
  least two positive curvature sites on the boundary $\bd D$ of $D$.
  If $D$ has no spurs and every boundary $2$-cell of $D$ has at least
  three nodes on its boundary then there are at least four positive
  curvature sites on $\bd D$, each of which consists of a $2$-cell with
  three nodes on its boundary.
\end{cor}
\begin{proof}
  An internal node $v$ of $D$ has $\rho(v) = \delta(v)$ and so has
  curvature
  \[ \kappa(v) = 2\pi - \delta(v)\pi + \frac{\delta(v)}{2}\pi = 2\pi -
  \frac{\delta(v)}{2}\pi, \]
  which, by the $\scT(4)$ property, is nonpositive.  For a boundary node
  $v$, $\rho(v) < \delta(v)$ so that,
  \[ \kappa(v) < 2\pi - \frac{\delta(v)}{2}\pi. \]
  But $\kappa(v)$ is an integer multiple of $\frac{\pi}{2}$, so if
  $\delta(v) \ge 3$ then $\kappa(v)$ is nonpositive.  But
  $\delta(v) \neq 0$, by assumption, and $\delta(v) \neq 2$, by
  \Defnref{arc}, so only if a node $v$ has valence $\delta(v) = 1$,
  i.e., $v$ is the tip of a spur, can it have positive curvature and
  in this case \[\kappa(v) = \pi.\]
  
  By assumption, $D$ is not a single $2$-cell and so $\kappa(F)$ is
  always less than $2\pi$ for a $2$-cell $F$ of $D$.  If $F$ has no
  boundary arc then, by the $\scC(4)$ property, $\kappa(F) \le 0$.  So
  only boundary $2$-cells can have positive curvature.  Then the only
  positive curvature in $D$ comes from boundary $2$-cells and tips of
  spurs, neither of which have curvature $2\pi$.  It follows, by
  \Propref{gbcftf}, that there are at least two sites of positive
  curvature on $\bd D$.
  
  If we add the assumption that $D$ has no spurs and all its boundary
  $2$-cells have at least three nodes then the only sites of positive
  curvature on $\bd D$ are boundary $2$-cells with exactly three nodes.
  These provide only $\frac{\pi}{2}$ curvature and so, by
  \Propref{gbcftf}, there must be at least four of them.
\end{proof}

\begin{prop}
  \proplabel{cftfemb} Let $Y$ be a simply connected $\cftf$ complex.
  Then every $2$-cell $F_1$ of $Y$ is embedded.
\end{prop}

\begin{prop}
  \proplabel{cftfint} Let $Y$ be a simply connected $\cftf$ complex
  and let $F_1$ and $F_2$ be a pair of intersecting $2$-cells of $Y$.
  Then the intersection $F_1 \cap F_2$ of these $2$-cells is connected.
\end{prop}

\begin{cor}
  \corlabel{cftfint} $2$-cells of a simply connected $\cftf$ complex $Y$
  intersect along paths, possibly closed, in its $1$-skeleton
  $\sk{1}{Y}$.
\end{cor}

\begin{prop}[Helly Property]
  \proplabel{helly} Let $F_1$, $F_2$ and $F_3$ be pairwise
  intersecting $2$-cells of a simply connected $\cftf$ complex $Y$.
  Then $F_1 \cap F_2 \cap F_3$ is nonempty.
\end{prop}

\newcommand{\drawsausages}[8]{
  \begingroup
  \pgfmathsetmacro\numsegments{#1}
  \pgfmathsetmacro\radius{#2}
  \pgfmathsetmacro\oradius{#3}
  \pgfmathsetmacro\iradius{#4}
  \pgfmathsetmacro\ocrado{#5}
  \pgfmathsetmacro\icrado{#6}
  \pgfmathsetmacro\opca{#7}
  \pgfmathsetmacro\ipca{#8}

  \node at (0,0) {$D$};

  \begin{scope}[thick]
    \foreach \i in {1,...,\numsegments}
    {
      \drawpetal{\i}{\i+1}{1}{\numsegments}{\radius}{\oradius}{\ocrado}{\opca}
      \drawpetal{\i}{\i+1}{1}{\numsegments}{\radius}{\iradius}{\icrado}{\ipca}
      \pgfmathsetmacro\angle{\i * 360 / \numsegments}
      \pgfmathsetmacro\tangle{\angle+90}
      \pgfmathsetmacro\offangle{\angle-180/\numsegments}
      \pgfmathsetmacro\oppoffangle{\offangle+180}
      \draw [decorate, decoration={markings,mark=at
        position \offangle/360 with
        {\arrow{>},\node[label={[label distance=-1mm]\oppoffangle:$\gamma_\i$}] {};}}]
      (0:\iradius) arc (0:360:\iradius);
      \node[vertex,label={\tangle:$v_\i$}] at (\angle:\radius) {};
      \node at (\offangle:\radius) {$F_\i$};
    }
  \end{scope}
  
  \endgroup
}

\begin{figure}
  \centering
  \begin{subfigure}[b]{0.33\textwidth}
    \centering
    \begin{tikzpicture}
      \pgfmathsetmacro\numsegments{1}
      \pgfmathsetmacro\radius{7/6}
      \pgfmathsetmacro\oradius{1.5 * \radius}
      \pgfmathsetmacro\iradius{2*\radius/3}
      \pgfmathsetmacro\ocrado{2.75*\radius}
      \pgfmathsetmacro\icrado{0}
      \pgfmathsetmacro\opca{65}
      \pgfmathsetmacro\ipca{55}

      \path[use as bounding box] (-\oradius, -\oradius) rectangle (\oradius, \oradius);
      \drawsausages{\numsegments}{\radius}{\oradius}{\iradius}{\ocrado}{\icrado}{\opca}{\ipca}
    \end{tikzpicture}
    \caption{\Propref{cftfemb}}
    \figlabel{disc1gon}
  \end{subfigure}%
  \begin{subfigure}[b]{0.33\textwidth}
    \centering
    \begin{tikzpicture}
      \pgfmathsetmacro\numsegments{2}
      \pgfmathsetmacro\radius{5/4}
      \pgfmathsetmacro\oradius{1.5 * \radius}
      \pgfmathsetmacro\iradius{2*\radius/3}
      \pgfmathsetmacro\ocrado{(1.75*\radius}
      \pgfmathsetmacro\icrado{\radius/3}
      \pgfmathsetmacro\opca{40}
      \pgfmathsetmacro\ipca{40}
      
      \path[use as bounding box] (-\oradius, -\oradius) rectangle (\oradius, \oradius);
      \drawsausages{\numsegments}{\radius}{\oradius}{\iradius}{\ocrado}{\icrado}{\opca}{\ipca}
    \end{tikzpicture}
    \caption{\Propref{cftfint}}
    \figlabel{disc2gon}
  \end{subfigure}%
  \begin{subfigure}[b]{0.33\textwidth}
    \centering
    \begin{tikzpicture}
      \pgfmathsetmacro\numsegments{3}
      \pgfmathsetmacro\radius{4/3}
      \pgfmathsetmacro\oradius{1.5 * \radius}
      \pgfmathsetmacro\iradius{2*\radius/3}
      \pgfmathsetmacro\ocrado{(\radius+3*\oradius)/4}
      \pgfmathsetmacro\icrado{\radius/3}
      \pgfmathsetmacro\opca{30}
      \pgfmathsetmacro\ipca{35}
      
      \path[use as bounding box] (-\oradius, -\oradius) rectangle (\oradius, \oradius);
      \drawsausages{\numsegments}{\radius}{\oradius}{\iradius}{\ocrado}{\icrado}{\opca}{\ipca}
    \end{tikzpicture}
    \caption{\Propref{helly}}
    \figlabel{disc3gon}
  \end{subfigure}
  \caption{Disc diagrams from the proofs of various propositions.}
  \figlabel{discngons}
\end{figure}

\begin{proof}[Proof scheme of Propositions \propref{cftfemb},
  \propref{cftfint} and \propref{helly}]
  Each proof follows the same pattern.  We assume the statement does
  not hold, and let the $F_i$, $\gamma_i$ and $v_i$ be a
  counterexample, where $v_i$ is a $0$-cell on the boundary of $F_i$
  and $F_{i+1}$ and $\gamma_i$ is an embedded path from $v_{i-1}$ to
  $v_i$ on the boundary of $F_i$.  Let $D$ be a disc diagram in $Y$
  with boundary path the concatenation of the $\gamma_i$, as in
  \Figref{discngons}.  Now, pick the counterexample
  $(F_i, \gamma_i, v_i)_i$ and $D$ so as to minimize the total number
  of cells in $D$.  Let $D' = D \cup \bigcup_i F_i$ be disc diagram
  obtained from $D$ by gluing in the $F_i$ along the $\gamma_i$, as in
  \Figref{discngons}.
  
  We claim that $D'$ is reduced.  If not then there would be a
  foldable arc $\alpha$ in some $\gamma_i$ between $F_i$ and some
  $2$-cell $F$ of $D$.  Then $\gamma_i$ is a concatenation
  $\beta_1\alpha\beta_2$ for some paths $\beta_1$ and $\beta_2$.  Let
  $\beta_1'$ be the path along $\bd F$ having the same length and
  terminal vertex as $\beta_1$ and having the same mapping as
  $\beta_1$ into $Y$.  Let $\beta_2'$ be the path along $\bd F$ having
  the same length and initial vertex as $\beta_1$ and having the same
  mapping as $\beta_1$ into $Y$.  We cut $D'$ open along the
  concatenations $\beta_1'\beta_1^{-1}$ and $\beta_2^{-1}\beta_2'$ and
  glue it back together to obtain a new disc diagram $D''$ in which
  $F$ and $F_i$ intersect along the entirety of $\gamma_i$.  But then
  the closure of $D'' \setminus F_i$ has fewer cells than $D'$,
  contradicting minimality of our choice.  So $D'$ is reduced.
  
  If $D$ is a single $0$-cell then the $F_i$ are not a counterexample,
  which is a contradiction.  Also $D$ cannot be a single $2$-cell
  since $D'$ is reduced and $Y$ is $\scC(4)$.  So $D$ is not a single
  $0$-cell or a single $2$-cell.  By minimality and the fact that
  $2$-cells are immersed, $D$ has no spurs.  Also, any boundary
  $2$-cell of $D$ that has fewer than three nodes would give a
  foldable arc of some $F_i$ in $D'$ so all boundary $2$-cells of $D$
  have at least three nodes.  Then, by \Corref{greendlingercftf},
  there are four sites of positive curvature on the boundary $\bd D$
  of $D$.  One of these must occur away from a $v_i$ and so give a
  foldable arc of an $F_i$ in $D'$, which again cannot be since $D'$
  is reduced.  So we have a contradiction.
\end{proof}

\begin{prop}[Strong Helly Property]
  \proplabel{stronghelly} Let $F_1$, $F_2$ and $F_3$ be pairwise
  intersecting $2$-cells of a simply connected $\cftf$ complex.  Then
  the intersection of a pair of the $2$-cells is contained in the
  remaining $2$-cell, i.e., \[F_{\sigma(1)} \cap F_{\sigma(2)} \subset
  F_{\sigma(3)}\] for some permutation $\sigma$ of the indices.
\end{prop}
\begin{proof}
  This follows immediately
  from \Propref{cftfemb}, \Corref{cftfint}, \Propref{helly} and the
  $\scT(4)$ property.
\end{proof}

\subsection{Quadrization of \texorpdfstring{$2$-Complexes}{2-Complexes}}
\sseclabel{quadrization}

We define now our main construction for this section, the quadrization
of a $2$-complex.  We use the properties developed above to study the
quadrization of $\cftf$ complexes and finally to prove that $\cftf$
groups are quadric.

Let $Y$ be a $2$-complex whose $2$-cells are embedded.  Let $Y_0$ and
$Y_2$ be the sets of $0$-cells and $2$-cells of $Y$.  Let $\Gamma_Y$ be
the bipartite graph on the vertex set $Y_0 \cup Y_2$ where an edge
joins $v \in Y_0$ and $F \in Y_2$ whenever $v$ appears on the boundary
of $F$.  The \defterm{quadrization} $X_Y$ of $Y$ is the $4$-flag
completion $X_Y = \overline \Gamma_Y$ of the graph $\Gamma_Y$.  Note
that if a group $G$ acts properly and cocompactly on $Y$ then it acts
properly and cocompactly on the quadrization $X_Y$.  If every $1$-cell
of $Y$ appears on the boundary of a $2$-cell then the quadrization $X_Y$
of $Y$ is connected.

For the rest of this section, we will assume that every $1$-cell of a
$2$-complex $Y$ is contained in the boundary of at least one $2$-cell.
This is not a serious restriction, since each $1$-cell of $Y$ not
appearing on the boundary of a $2$-cell can be subdivided (if it is not
embedded) and then thickened to a $2$-cell to obtain a new $2$-complex
$Y'$ which deformation retracts to $Y$.  The original complex $Y$
embeds $\Aut(Y)$-equivariantly into $Y'$ via a continuous map $Y \to
Y'$, which thus faithfully preserves group actions.  Furthermore, if
$Y$ is $\cftf$ then so is $Y'$.

\begin{figure}
  \centering
  \begin{tikzpicture}
    \pgfmathsetmacro\sp{15}
    \pgfmathsetmacro\irad{2.5}
    \pgfmathsetmacro\mrad{3.25}
    \pgfmathsetmacro\orad{4}
    \pgfmathsetmacro\crad{0.75}

    \path[use as bounding box]
      (-\irad/2-\mrad/2, -2) rectangle (\irad/2+\mrad/2, \orad+2/10);
    
    \begin{scope}[thick]
      \foreach \i in {-3,-1,1} {
        \pgfmathsetmacro\pango{90 + \i*\sp}
        \pgfmathsetmacro\pangt{90 + (\i+1)*\sp}
        \pgfmathsetmacro\pangr{90 + (\i+2)*\sp}
        
        \draw (\pango:\mrad)
          to[out=\pango, in=\pangt-90] (\pangt:\orad)
          to[out=\pangt+90, in=\pangr] (\pangr:\mrad)
          to[out=\pangr+180, in=\pangt+90] (\pangt:\irad)
          to[out=\pangt-90, in=\pango+180] (\pango:\mrad);
      }
      
      \draw[decorate,
        decoration={markings,mark=at position 1/2 with {\arrow{>}}}]
        (90+\sp:\mrad) to[out=90+15+180, in=90+90] (90:\irad);

      \draw[decorate,
        decoration={markings,mark=at position 1/2 with {\arrow{>>}}}]
        (90:\irad) to[out=90-90, in=90-\sp+180] (90-\sp:\mrad);

      \draw (90:\irad) arc (90:90+360:\crad);
      \node at (90:\irad-\crad) {$F$};
      
      \node at (90-2*\sp:\mrad) {$F_{2i+2}$};
      \node at (90:\mrad) {$F_{2i}$};
      \node at (90+2*\sp:\mrad) {$F_{2i-2}$};

      \node at (90+\sp:\orad+1/10) {$v_{2i-1}$};
      \node at (90-\sp:\orad+1/10) {$v_{2i+1}$};
      \node at (90:\irad-3/10) {$w$};

      \draw (90+3*\sp:\mrad)
        to[out=90+3*\sp+90, in=180] (0,-2)
        to[out=0, in=90-3*\sp-90] (90-3*\sp:\mrad);
      
      \node at (0,0) {$D \setminus F$};
    \end{scope}
    
    \node[vertex] at (90:\irad) {};
    \node[vertex] at (90-3*\sp:\mrad) {};
    \node[vertex] at (90-1*\sp:\mrad) {};
    \node[vertex] at (90+1*\sp:\mrad) {};
    \node[vertex] at (90+3*\sp:\mrad) {};
  \end{tikzpicture}%
  \begin{tikzpicture}
    \pgfmathsetmacro\sp{15}
    \pgfmathsetmacro\irad{2.5}
    \pgfmathsetmacro\mrad{3.25}
    \pgfmathsetmacro\orad{4}
    \pgfmathsetmacro\oorad{4.5}
    \pgfmathsetmacro\crad{0.75}

    \path[use as bounding box]
      (-\irad/2-\mrad/2, -2) rectangle (\irad/2+\mrad/2, \oorad);
    
    \begin{scope}[thick]
      \foreach \i in {-3,1} {
        \pgfmathsetmacro\pango{90 + \i*\sp}
        \pgfmathsetmacro\pangt{90 + (\i+1)*\sp}
        \pgfmathsetmacro\pangr{90 + (\i+2)*\sp}
        
        \draw (\pango:\mrad)
          to[out=\pango, in=\pangt-90] (\pangt:\orad)
          to[out=\pangt+90, in=\pangr] (\pangr:\mrad)
          to[out=\pangr+180, in=\pangt+90] (\pangt:\irad)
          to[out=\pangt-90, in=\pango+180] (\pango:\mrad);
      }
      
      \draw (90+\sp:\mrad)
        to[out=90+\sp+180, in=90+90] (90:\irad)
        to[out=90, in=30] (90+3*\sp/4:\oorad) node[vertex] {}
        to[out=210, in=90+\sp] (90+\sp:\mrad);

      \draw (90-\sp:\mrad)
        to[out=90-\sp+180, in=90-90] (90:\irad)
        to[out=90, in=150] (90-3*\sp/4:\oorad) node[vertex] {}
        to[out=330, in=90-\sp] (90-\sp:\mrad);

      \draw[decorate,
        decoration={markings,mark=at position 1/2 with {\arrow{>}}}]
        (90+\sp:\mrad) to[out=90+15+180, in=90+90] (90:\irad);

      \draw[decorate,
        decoration={markings,mark=at position 1/2 with {\arrow{>>}}}]
        (90:\irad) to[out=90-90, in=90-\sp+180] (90-\sp:\mrad);

      \draw[decorate,
        decoration={markings,mark=at position 3/5 with {\arrow{>>}}}]
        (90:\irad) to[out=90, in=30] (90+3*\sp/4:\oorad);

      \draw[decorate,
        decoration={markings,mark=at position 3/5 with {\arrow{<}}}]
        (90:\irad) to[out=90, in=150] (90-3*\sp/4:\oorad);

      \draw (90:\irad) arc (90:90+360:\crad);
      \node at (90:\irad-\crad) {$F$};
      
      \node at (90-2*\sp:\mrad) {$F_{2i+2}$};
      \node at (90+2*\sp:\mrad) {$F_{2i-2}$};

      \node at (90-4*\sp/7:\mrad/2+\orad/2) {$F_{2i}$};
      \node at (90+4*\sp/7:\mrad/2+\orad/2) {$F_{2i}$};
      
      \node at (90+6*\sp/4:\orad+1/10) {$v_{2i-1}$};
      \node at (90-6*\sp/4:\orad+1/10) {$v_{2i+1}$};
      \node at (90:\irad-3/10) {$w$};

      \draw (90+3*\sp:\mrad)
        to[out=90+3*\sp+90, in=180] (0,-2)
        to[out=0, in=90-3*\sp-90] (90-3*\sp:\mrad);
      
      \node at (0,0) {$D \setminus F$};
    \end{scope}
    
    \node[vertex] at (90:\irad) {};
    \node[vertex] at (90-3*\sp:\mrad) {};
    \node[vertex] at (90-1*\sp:\mrad) {};
    \node[vertex] at (90+1*\sp:\mrad) {};
    \node[vertex] at (90+3*\sp:\mrad) {};
  \end{tikzpicture}
  \caption[Two disc diagrams from the proof of \Lemref{quadsc}.]{The
    disc diagram $D$ from the proof of \Lemref{quadsc}.  If, as on the
    left, it has a $2$-cell $F$ then the disc diagram on the right
    exists, contradicting minimality of $D$.}
  \figlabel{quadsc}
\end{figure}

\begin{lem}
  \lemlabel{quadsc} Let $Y$ be a $2$-complex with embedded $2$-cells.  If
  $Y$ is simply connected then so is its quadrization $X_Y$.
\end{lem}
\begin{proof}
  Suppose $X_Y$ is not simply connected.  Let $\alpha$ be any closed
  essential path $\alpha$ of $X_Y$.  In $Y$, $\alpha$ corresponds to a
  cyclic sequence $F_0$, $v_1$, $F_2$, $v_3$, \ldots, $F_{2n-2}$,
  $v_{2n-1}$, $F_0$ of $2$-cells and $0$-cells such that, for every $i$,
  the $2$-cell $F_{2i}$ contains the $0$-cells $v_{2i-1}$ and $v_{2i+1}$
  on its boundary (indices modulo $2n$).  Let $\beta_{2i}$ and
  $\gamma_{2i}$ be the two paths on the boundary of $F_{2i}$ from
  $v_{2i-1}$ to $v_{2i+1}$.  Let $\delta_0$, $\delta_2$, \ldots,
  $\delta_{2n-2}$ be any sequence with each
  $\delta_{2i} \in \{ \beta_{2i}, \gamma_{2i} \}$.  Then the
  concatenation $\delta = \delta_0 \delta_2 \cdots \delta_{2n-2}$ is a
  closed path in $Y$.  So there is a disc diagram $D$ in $Y$ with
  boundary path $\pbd D = \delta$.
  
  Now, choose $\alpha$, the $\delta_{2i}$ and $D$ so as to minimize
  the total number of $0$-cells, $1$-cells and $2$-cells of $D$.  Because
  the $2$-cells of $Y$ are embedded, any tip of a spur of $D$ must be
  one of the $0$-cells $v_{2i+1}$.  Let $u$ be the $0$-cell of the spur
  incident to this $v_{2i+1}$.  Then $F_{2i}$, $v_{2i+1}$, $F_{2i+2}$,
  $u$, $F_{2i}$ is a $4$-cycle in $X_Y$ and so is nullhomotopic.  But
  then we can replace $v_{2i+1}$ with $u$ in $\alpha$ and $D$ with
  $D \setminus e$, where $e$ is the $1$-cell of the spur joining $u$ and
  $v_{2i+1}$.  This contradicts the minimality of our choices and so
  $D$ has no spurs.
  
  Suppose $D$ has a $2$-cell.  Then $D$ must have a $2$-cell $F$ that
  intersects its boundary $\bd D$.  Let $w$ be a $0$-cell in the
  intersection of $F$ and $\bd D$.  Then $w$ is contained in some
  $\delta_{2i}$ and so some $F_{2i}$ in $X_Y$.  But then we can
  replace the subpath $F_{2i}$ in $\alpha$ with $F_{2i}$, $w$, $F$,
  $w$, $F_{2i}$ and replace $D$ with $D \setminus F$, as shown in
  \Figref{quadsc}.  This contradicts the minimality of our choices.
  So $D$ has no $2$-cells.
  
  Therefore, $D$ is a spurless tree, i.e., a single $0$-cell, $x$.  It
  follows that $v_{2i+1} = x$, for every $i$, so that $\alpha$ is
  $F_0$, $x$, $F_2$, $x$, \ldots, $F_{2n-2}$, $x$, $F_0$.  This path
  is clearly nullhomotopic, which is a contradiction.
\end{proof}

\begin{lem}
  \lemlabel{cftfq} Let $Y$ be a simply connected $2$-complex.  If $Y$ is
  $\cftf$ then its quadrization $X_Y$ is quadric.
\end{lem}
\begin{proof}
  By \Propref{cftfemb}, the $2$-cells of $Y$ are embedded, so by
  \Lemref{quadsc}, $X_Y$ is simply connected.  Then, by
  \Propref{mgchar}, it suffices to show that embedded $6$-cycles of
  $\sk{1}{X_Y}$ have diagonals.  An embedded $6$-cycle in $X_Y$
  corresponds to a triple of pairwise intersecting $2$-cells $F_1$,
  $F_2$ and $F_3$ in $Y$ and three $0$-cells, one contained in each of
  the three pairwise intersections.  By \Propref{stronghelly},
  $F_1 \cap F_2 \subset F_3$, after possibly reindexing.  Then the
  $0$-cell contained in $F_1 \cap F_2$ is incident to $F_3$.  Hence
  the $6$-cycle has a diagonal.
\end{proof}

\begin{thm}
  \thmlabel{cftfquad} Let $G$ be a group acting properly and
  cocompactly on a simply connected $\cftf$ $2$-complex $Y$.  Then $G$
  is quadric.
\end{thm}
\begin{proof}
  The quadrization $X_Y$ of $Y$ is quadric, by \Lemref{cftfq}.  Since
  the action of $G$ on the $0$-cells and $2$-cells of $Y$ preserves
  incidences, $G$ acts on $X_Y$.  The stabilizers of $0$-cells and
  $2$-cells are trivial in $Y$, so then are the stabilizers of vertices
  and edges in $X_Y$.  Then the action must be proper since the
  stabilizer of a square contains the pointwise stabilizer as a finite
  index subgroup and the pointwise stabilizer must be trivial since it
  stabilizes the vertices of the square.  This action is also
  cocompact and so we are done.
\end{proof}

\begin{cor}
  \corlabel{cftfpresquad} Let $G$ be a group admitting a finite $\cftf$
  presentation $\gpres{\mcS}{\mcR}$.  Then $G$ is quadric.
\end{cor}
\begin{proof}
  The action of $G$ on the universal cover $\ucov{Y}$ of the complex
  $Y$ associated to the presentation $\gpres{\mcS}{\mcR}$ is free and
  cocompact.  Hence $G$ is quadric by \Thmref{cftfquad}.
\end{proof}

We include the following corollary which essentially presents the
$\cftf$ case of Huebschmann's theorem on the finite subgroups of small
cancellation groups \cite{Huebschmann:1979}.  Huebschmann's proof
originally relied on a theorem of a paper of Lyndon whose correct
proof was given later by Collins and Huebschmann \cite{Collins:1982}.
The corollary below does not rely on cohomological arguments, as does
Huebschmann's proof, so we believe it may be of independent interest.

\begin{cor}
  Let $G$ be a finite group acting on a simply connected, locally
  finite $\cftf$ $2$-complex $Y$.  Then $G$ stabilizes a $0$-cell, a
  $1$-cell or the boundary of a $2$-cell of $Y$.
\end{cor}
\begin{proof}
  We may assume that $Y$ is not a single $0$-cell so that the
  quadrization $X_Y$ of $Y$ is not a single vertex.  By \Thmref{ibt},
  there is an invariant biclique of $X_Y$.  Let
  $\{v_i\}_{i=1}^{k} \cup \{F_j\}_{j=1}^{\ell}$ be the vertex set of
  this invariant biclique with the $v_i$ corresponding to $0$-cells of
  $Y$ and the $F_j$ corresponding to $2$-cells of $Y$.  Then
  $\bigcap_{j=1}^{\ell} F_j$ is invariant and
  $\{v_i\}_{i=1}^{k} \subset \bigcap_{j=1}^{\ell} F_j$ so
  $\bigcap_{j=1}^{\ell} F_j$ is nonempty.  By applying
  \Propref{stronghelly} inductively to $\bigcap_{j=1}^{\ell} F_j$ we
  see that $\bigcap_{j=1}^{\ell} F_j = F_j \cap F_{j'}$ for some $j$
  and $j'$.  So, by \Propref{cftfemb} and \Propref{cftfint}, either
  $G$ stabilizes the boundary of $F_j$ or $G$ stabilizes an embedded
  subpath $P$ of the boundary of $F_j$.  In the former case we are
  done and in the latter case $G$ stabilizes the midpoint of $P$
  which, depending on the parity of $|P|$, is either contained in an
  invariant $0$-cell or an invariant $1$-cell.
\end{proof}

\bibliographystyle{abbrv}
\bibliography{nima}
\end{document}